\documentclass[10pt]{article}
\usepackage{mlpackage}
\usepackage[backend=biber,
            style=authoryear,
            sorting=nyt,
            uniquename=true
           ]{biblatex}
\providetoggle{arxiv}
\settoggle{arxiv}{true}

\DeclareCiteCommand{\cite}
  {\usebibmacro{prenote}}
  {\usebibmacro{citeindex}%
   \printtext[bibhyperref]{\usebibmacro{cite}}}
  {\multicitedelim}
  {\usebibmacro{postnote}}

\DeclareCiteCommand*{\cite}
  {\usebibmacro{prenote}}
  {\usebibmacro{citeindex}%
   \printtext[bibhyperref]{\usebibmacro{citeyear}}}
  {\multicitedelim}
  {\usebibmacro{postnote}}

\DeclareCiteCommand{\parencite}[\mkbibparens]
  {\usebibmacro{prenote}}
  {\usebibmacro{citeindex}%
    \printtext[bibhyperref]{\usebibmacro{cite}}}
  {\multicitedelim}
  {\usebibmacro{postnote}}

\DeclareCiteCommand*{\parencite}[\mkbibparens]
  {\usebibmacro{prenote}}
  {\usebibmacro{citeindex}%
    \printtext[bibhyperref]{\usebibmacro{citeyear}}}
  {\multicitedelim}
  {\usebibmacro{postnote}}

\DeclareCiteCommand{\footcite}[\mkbibfootnote]
  {\usebibmacro{prenote}}
  {\usebibmacro{citeindex}%
  \printtext[bibhyperref]{ \usebibmacro{cite}}}
  {\multicitedelim}
  {\usebibmacro{postnote}}

\DeclareCiteCommand{\footcitetext}[\mkbibfootnotetext]
  {\usebibmacro{prenote}}
  {\usebibmacro{citeindex}%
   \printtext[bibhyperref]{\usebibmacro{cite}}}
  {\multicitedelim}
  {\usebibmacro{postnote}}

\DeclareCiteCommand{\textcite}
  {\boolfalse{cbx:parens}}
  {\usebibmacro{citeindex}%
   \printtext[bibhyperref]{\usebibmacro{textcite}}}
  {\ifbool{cbx:parens}
     {\bibcloseparen\global\boolfalse{cbx:parens}}
     {}%
   \multicitedelim}
  {\usebibmacro{textcite:postnote}}

\DeclareMathOperator{\segaux}{seg}                      
\newcommand*{\conn}{\nabla}                             
\newcommand*{\lie}{\mathcal{L}}                         
\newcommand*{\connflat}{\nabla^{\text{flat}}}           

\newcommand*{\gm}{\textsl{g}}
\newcommand*{\paral}[1]{#1^{\dag}}                      
\newcommand*{\normal}[1]{#1^{\perp}}                    
\newcommand*{\dgamma}{\dot{\gamma}}                     
\DeclareMathOperator{\sn}{sn}                           
\DeclareMathOperator{\ct}{ct}                           
\newcommand*{\tseg}{\overline{U}_p}                     
\newcommand*{\seg}{U_p}                                 
\newcommand*{\segall}{\segaux(p)}                       
\DeclareMathOperator{\conjloc}{conj}                    
\DeclareMathOperator{\cut}{cut}                         
\DeclareMathOperator{\inj}{r_{inj}}                     
\DeclareMathOperator{\diam}{diam}                       
\newcommand\glie{\mathfrak{g}}                          
\newcommand\hlie{\mathfrak{h}}                          
\newcommand\mlie{\mathfrak{m}}                          
\newcommand\transaux{\intercal}                         
\newcommand\trans[1]{#1^\transaux}                      
\renewcommand*\SS{\mathbb{S}}                           

\newcommand*\rgd{\textsc{rgd}}

\newcommand*\PD{\textsc{PD}}

\DeclareMathOperator{\SOaux}{SO}                         
\DeclareMathOperator{\soaux}{\mathfrak{so}}              
\DeclareMathOperator{\Skewaux}{Skew}                     
\DeclareMathOperator{\Graux}{Gr}                         

\NewDocumentCommand{\SO}{ m }{ \SOaux\pa{#1} }

\NewDocumentCommand{\so}{ m }{ \soaux\pa{#1} }

\NewDocumentCommand{\Skew}{ m }{ \Skewaux\pa{#1} }

\NewDocumentCommand{\M}{ >{\SplitArgument{1}{,}}m}{%
    \RR^{\prodaux #1}%
}

\NewDocumentCommand{\Gr}{ >{\SplitArgument{1}{,}}m}{%
    \Graux\pa{\commasaux #1}%
}

\NewDocumentCommand{\commasaux}{ m m }{%
    \IfNoValueTF{#2}{ #1 }{ #1, #2 }%
}
\NewDocumentCommand{\prodaux}{ m m }{%
    \IfNoValueTF{#2}{ #1 \times #1 }{ #1 \times #2 }%
}

\NewDocumentCommand{\liebrack}{s O{} >{\SplitArgument{1}{,}}m}{%
    \IfBooleanTF{#1}{\liebrackaux*#3}{\liebrackaux[#2]#3}%
}
\DeclarePairedDelimiterX{\liebrackaux}[2]{\lbrack}{\rbrack}{#1, #2}

\usepackage{todonotes}

\title{Curvature-Dependant Global Convergence Rates for Optimization on Manifolds of Bounded Geometry}
\author{Mario Lezcano-Casado\thanks{Mathematical Institute, University of Oxford, United Kingdom.
\newline \hspace*{4.4ex}Correspondence: mario.lezcanocasado@maths.ox.ac.uk}}
\date{}

\begin{document}
\maketitle
\begin{abstract}
    We give curvature-dependant convergence rates for the optimization of weakly convex functions defined on a manifold of $1$-bounded geometry via Riemannian gradient descent and via the \emph{dynamic trivialization algorithm}. In order to do this, we give a tighter bound on the norm of the Hessian of the Riemannian exponential than the previously known. We compute these bounds explicitly for some manifolds commonly used in the optimization literature such as the special orthogonal group and the real Grassmannian. Along the way, we present self-contained proofs of fully general bounds on the norm of the differential of the exponential map and certain cosine inequalities on manifolds, which are commonly used in optimization on manifolds.
\end{abstract}

\iftoggle{arxiv}{%
\section{Introduction}
We are interested in approximating critical points of a---possibly non-convex---function defined on a manifold
\[
    \min_{x \in M} f(x).
\]
In particular, we are interested in giving global convergence bounds for such problems. These problems form a particularly well-behaved subset of the field of constrained optimization. Examples of the importance of this family of problems are ubiquitous in engineering, statistics, and machine learning. For example, they have found applications in image processing with the Grassmannian for tracking subspaces; analysis of time series in deep learning with the orthogonal group to avoid vanishing and exploding gradient problems; Bayesian statistics and machine learning with the manifold of positive definite (\PD) matrices in kernel methods and metric learning; the space of low-rank and fixed-rank matrices for low-rank models; and the hyperbolic space for word embeddings, among many others.

The approach that we investigate in this paper is that of using the exponential map to pullback the problem from the whole manifold to a fixed tangent space, converting the problem into an unconstrained one over $T_pM \iso \RR^n$
\begin{equation}\label{eq:problem}
    \min_{v \in T_pM} f(\exp_p(v)).
\end{equation}
For this to be a sensible method, the exponential map has to be surjective, so we will assume $(M, \gm)$ to be \textbf{connected and complete} throughout the paper.

\paragraph{General pullbacks in constrained optimization.}
The idea of pulling back a problem from a space with a complex geometry to a simpler one via a particular family of maps lies at the core of the area of constrained optimization. Mirror descent~\parencite{nemirovsky1983problem} and Lagrange multipliers are examples of methods that allow simplifying a constrained optimization problem into one in a simpler space. It is also a recurrent theme in the literature on optimization on manifolds. We have the Burer-Monteiro method~\parencite{burer2003nonlinear,burer2005local} which parametrizes the space of low-rank matrices via the multiplication of two tall matrices. Other approaches via different factorizations have also been explored, for example via an SVD factorization~\parencite{vandereycken2013low} or other factorizations like the polar factorization~\parencite{mishra2014fixed}. Similar parametrizations have been explored for the algebraic variety of matrices of rank at most $k$---low-rank optimization---and manifolds such as positive semidefinite matrices of a fixed rank~\parencite{vandereycken2013riemannian,massart2020quotient}. The idea behind these algorithms is that of parametrizing a space in terms of simpler ones, effectively pulling back the problem to an easier one and then mapping back the solution of the simpler problem to the complex space via the parametrization map.

\paragraph{Pullbacks along the exponential in numerical analysis.}
Instances of~\eqref{eq:problem} can be found all throughout the area of numerical analysis. For example, in the study of the center of mass in the manifold of \PD{} matrices~\parencite{arsigny2006log,arsigny2006geometric}; continuous dynamics on Lie groups, and their discretizations~\parencite{magnus1954exponential,iserles1999solution,iserles2000lie}; optimization on compact Lie groups~\parencite{lezcano2019cheap}; and probabilistic methods on manifolds for their use in machine learning~\parencite{farlosi2019reparametrizing}. More recently, these ideas were unified into the framework of \emph{static trivializations}~\parencite{lezcano2019trivializations}.

The main idea behind pulling back an optimization problem from a manifold to a flat space is that of heavily simplifying the problem at hand, going from a manifold which might have a complex topology and geometry, to one that is flat and has trivial topology. Of course, that comes at a cost, as this change of topology makes the new function $f \circ \exp_p$ have more critical points. Luckily, for any differentiable manifold and any smooth metric, this set of problematic points, called the \emph{conjugate locus}, has measure zero by a theorem of Sard~\parencite{sard1965hausdorff}. Since the conjugate locus is a subset of the \emph{cut locus}, it is far from the point $p$. This is because the cut locus can vaguely be regarded as \emph{the set that is opposite to the point $p$}. For example, on a sphere, the cut locus of a point $p$ is its antipode and on a cylinder, the cut locus of a point is the whole line opposite to that point. As such, the cut locus is, in some sense, \emph{as far from $p$ as possible}, and so is the conjugate locus. At the expense of this technicality, one may heavily simplify the problem of optimizing a function on a space with a non-trivial topology to a Euclidean problem. We will expand on this technical point in~\Cref{sec:background}.

\paragraph{Pullbacks along the exponential map: Trivializations.}
If we pullback the problem to the tangent space at each iteration of the algorithm along a retraction, we recover Riemannian Gradient Descent (\rgd) along that retraction~\parencite{boumal2019global}. It was then noted in~\parencite{lezcano2019trivializations} that one may change the point $p$ in~\eqref{eq:problem} to avoid converging to a point in the conjugate locus of $p$. In fact, one may change the point according to an arbitrary stopping rule giving the following meta-algorithm:

\begin{enumerate}
    \item Compute gradient steps of the function $f\circ \exp_{p_i}$ to obtain iterates $v_{i,k}\in T_{p_i}M$ for $k = 1, 2, \dots$
    \item If we are close to a point in the conjugate locus of $p_i$, we set $p_{i+1} = \exp_{p_i}(v_{i,k})$ and repeat
\end{enumerate}

This translates to~\Cref{alg:dyn_triv} given a boolean stopping rule $\code{stop}$ and it is called the \emph{dynamic trivializations framework}~\parencite{lezcano2019trivializations}.
\begin{algorithm}[!htp]
    \caption{Dynamic trivialization framework}
    \label{alg:dyn_triv}
    \begin{algorithmic}[1]
            \Require A starting point $p_0 \in M$, a boolean statement $\code{stop}$
            \For {$i = 0, \ldots$}
                \State $v_{i, 0} = 0$
                \For {$k = 1, \dots$}
                    \State $v_{i, k} = v_{i, k-1} - \eta_{i,k-1}\grad \pa{f \circ \exp_{p_i}} \pa{v_{i,k-1}}$\label{line:update}\Comment{Optimize on $T_{p_i}M$}
                    \If {\code{stop}}\Comment{Stopping condition}
                        \Break
                    \EndIf
                \EndFor
                \State $p_{i+1} = \exp_{p_i}(v_{i,k})$ \Comment{Update pullback point}
            \EndFor
    \end{algorithmic}
\end{algorithm}

Two possible stopping rules stand out over the others. If we let $\code{stop} \equiv \code{False}$, then we recover the \emph{static-trivialization framework} of pulling back the function to a fixed tangent space as in~\eqref{eq:problem}. On the other hand, if $\code{stop} \equiv \code{True}$, then the algorithm is exactly Riemannian gradient descent, as mentioned before. The stopping rule of being too close to the conjugate locus could be then written as
\[
    \code{stop} \equiv \frac{\norm{\grad\pa{f \circ \exp_{p_i}}\pa{v_{i, k}}}}{\norm{\grad f\pa{x_{i,k}}}} < \epsilon.
\]

The ideas of pulling back the problem at the current point of the optimization and then taking a few steps on the tangent space have been recently explored in~\parencite{criscitiello2019efficiently,sun2019escaping} in the context of escaping saddle points for optimization on manifolds.

    \paragraph{Retractions with bounded second and third order derivatives.}
    To be be able to prove convergence of the dynamic trivialization framework, we will give novel bounds on the norm of the Hessian of the exponential map in terms of the curvature of the manifold. These bounds are one example of second order bounds for a large family of retractions. These bounds lie at the heart of the convergence results of Riemannian methods via retractions, where it is often referred to as the \emph{$L$-smoothness of the retraction}. Examples of work under this assumptions are adaptive methods on matrix manifolds~\parencite{kasai2019adaptive}, quasi-Newton methods on manifolds~\parencite{wen2015broyden}, stochastic methods with variance reduction~\parencite{kasai2018riemannian}, and general convergence of Newton methods in Riemannian manifolds~\parencite{ferreira2002kantorovich}, among many others. The bounds developed in this paper give exact rates of growth for the norm of the Hessian of the retraction in terms of the curvature of the manifold when the retraction is the exponential map, which falls exactly in the framework of these papers and many others.

    Third order and higher order bounds are used when trying to converge to second order optima and escaping saddle points~\parencite{criscitiello2019efficiently} and when adapting Newton methods to manifolds~\parencite{agarwal2020arc}, and other higher order methods. The computations of higher order bounds is analogous to that of second order bounds. In this paper, we present all the necessary techniques to give $n$-th order bounds for a manifold of $n$-bounded geometry (see~\Cref{def:bounded_geometry} for a definition), although we do not perform these computations explicitly. Bounds of this flavour were first developed in~\parencite{eichhorn1991boundedness}, where the rate of growth of the $n$-th derivative of the Christoffel symbols in normal coordinates is estimated.

    \paragraph{General assumptions.}
    In this paper, we look at the problem of proving convergence for the family of dynamic trivialization methods all at once. To do so, it is enough to control the iterates of the (Euclidean) gradient descent in the context of problem~\eqref{eq:problem}. If the bounds do not have a dependency on $p$, then they will provide guarantees for~\Cref{alg:dyn_triv}. This is the road that we will take in this paper.

    For some manifolds, such as the hyperbolic space or the space of symmetric positive definite matrices (and more generally on Hadamard spaces), one may find (geodesically) convex functions which can be optimized efficiently~\parencite{bacak2014convex}. On the other hand, when the manifold is compact, as is the case of the sphere, or when dealing with orthogonality constraints~\parencite{edelman1998geometry}, there exists no convex function besides the constant ones, so we inevitably have to deal with non-convex problems.

    To this end, we will consider a function $f$ which is \textbf{$\alpha$-weakly convex} (\ie, with bounded Hessian in operator norm) and we ask ourselves whether $f \circ \exp_p$ is $\alpha'$-weakly convex for any $\alpha'$. The answer to this problem will be: Not in general. For example, for the hyperbolic space, the exponential map grows as $\cosh(t)$, which has unbounded first and second derivatives. For this reason, we will have to content ourselves with proving tight bounds on the growth of the Hessian of $f \circ \exp_p$ on a neighbourhood of a given radius.

    When it comes to the assumptions on the manifold, we will look at the family of \textbf{manifolds of bounded first order geometry}.
    \begin{definition}[Bounded first order geometry]
        We say that a Riemannian manifold $(M, \gm)$ has $(\delta, \Delta, \Lambda)$-bounded geometry if it has uniformly bounded injectivity radius $\inj > 0$ and we have bounds on the sectional curvature and the derivative of the curvature tensor $R$
        \[
            \delta \leq \sec \leq \Delta \qquad \norm{\conn R} < \Lambda.
        \]
    \end{definition}
    These manifolds have found many uses in the area of PDEs on manifolds and geometric analysis. In particular, any compact manifold with any metric is of bounded geometry. On the other hand, this family does not impose topological restrictions on $M$ as any differentiable manifold admits a complete metric of bounded geometry~\parencite{greene1978complete}.

    For some general introduction to this family of manifolds see the thesis~\parencite{eldering2012persistence}.

    \subsection{Contributions}
    In the paper, we will implement the following strategy:
    \begin{enumerate}
        \item Give first and second order bounds for $\exp_p$ that depend on the curvature of the manifold
        \item Use these bounds to estimate the weak convexity of $f \circ \exp_p$ on a bounded domain
        \item Use these weak convexity bounds to prove the convergence of the dynamic trivialization framework
    \end{enumerate}

    First order bounds for the Riemannian exponential are well known in geometry under the name of \emph{Rauch's theorem}. We will present a self-contained proof of them here in full generality in~\Cref{sec:first_order_bounds} as it is not easily found in the literature. We also take the chance to prove a law of cosines in manifolds that follows from this in~\Cref{sec:law_of_cosines}. Particular cases of this law of cosines have been used throughout the optimization literature, but it is not easy to find a self-contained proof of the general case.

    Second order bounds for the exponential are considerably more difficult to prove. There have been some recent attempts at bounding these quantities in the context of optimization on manifolds~\parencite{sun2019escaping}. In particular, concurrent to this work, some new bounds were given in~\parencite{chriscitiello2020accelerated}. The bounds that we present in this work are tighter than those in these papers. Historically, Kaul was the first to prove second order bounds in~\parencite{kaul1976schranken}. We borrow some ideas from this paper, while greatly improving their bounds. In particular, these original bounds were given in terms of a solution of a differential equation that cannot be integrated. The bounds that we give in~\Cref{sec:second_order_bounds} are explicit in terms of trigonometric functions.
    \begin{theorem}[Bounds on the Full Hessian]
        Let $(M, \gm)$ be a Riemannian manifold with $(\delta, \Delta, \Lambda)$-bounded geometry. For a geodesic $\deffun{\gamma : [0,r] -> M;}$ with initial unit vector $v$, $r < \pi_{\frac{\Delta + \delta}{2}}$, and any two vectors $w_1, w_2 \in T_pM$, we have that
        \[
            \norm{\pa{\conn\dif \exp_p}_{rv}\pa{w_1,w_2}} \leq
                \lfrac{8}{3r^2}\sn_\delta\pa[\big]{\lfrac{r}{2}}^2
                \pa{\Lambda\sn_\delta\pa[\big]{\lfrac{r}{2}}^2
                +2\max\set{\abs{\Delta}, \abs{\delta}}\sn_\delta\pa{r}}\norm{w_1}\norm{w_2}.
        \]
        where
        \[
            \sn_\delta(t) \defi
            \begin{cases}
                \frac{\sin (\sqrt{\delta}t)}{\sqrt{\delta}}    \qquad &\text{if } \delta > 0\\
                t                                              \qquad &\text{if } \delta = 0\\
                \frac{\sinh (\sqrt{-\delta}t)}{\sqrt{-\delta}} \qquad &\text{if } \delta < 0
            \end{cases}
        \]
        and $\pi_\kappa \in (0, \infty]$ is the first positive zero of $\sn_\kappa$.
        Furthermore, the radius $\pi_{\frac{\Delta + \delta}{2}}$ is tight for $\SO{n}$ with a bi-invariant metric.
    \end{theorem}
    We also give tighter bounds in~\Cref{thm:second_order_bounds}, although we believe that this more compact bound will be easier to handle in general applications.

    We compute explicit bounds for some manifolds which find uses in optimization in~\Cref{sec:concrete_bounds}. For example, for the special orthogonal embedded in $\RR^{n \times n}$ group we get
    \[
        \norm{\pa{\conn \dif \exp_p}_{rv}(w_1,w_2)} \leq \frac{r}{3}\norm{w_1}\norm{w_2} \qquad{r < 2\sqrt{2}\pi}
    \]
    and for the real Grassmannian with the canonical metric
    \[
        \norm{\pa{\conn \dif \exp_p}_{rv}(w_1,w_2)} \leq \frac{8r}{3}\norm{w_1}\norm{w_2} \qquad{r < \pi}.
    \]

    Once we have first and second order bounds, in~\Cref{sec:weakly_convex}, we bound the differential and the Hessian of the exponential map using the chain rule and Cauchy-Schwartz
    \[
        \norm{\Hess (f \circ \exp_p)} \leq \norm{\Hess f}\norm{\dif \exp_p}^2 + \norm{\dif f}\norm{\Hess \exp_p}
    \]
    getting the following result.
    \begin{theorem}[Weak convexity of the pullback]
        Let $(M, \gm)$ be a connected and complete Riemannian manifold and let $f$ be an $\alpha$-weakly convex function on it. Fix a point $p \in M$, and a set $\overline{\mathcal{X}} \subset T_pM$ with $\diam(\overline{\mathcal{X}}) = r < \pi_{\frac{\Delta+\delta}{2}}$ that does not intersect the conjugate locus and has at least one critical point of $f$ in it. Then the map $f \circ \exp_p$ is $\widehat{\alpha}_r$-weakly convex with constant
        \[
            \widehat{\alpha}_r = \alpha (C_{1,r} + C_{2,r})
        \]
        for
        \begin{gather*}
            C_{1,r} = \max\set[\Big]{1, \lfrac{\sn_\delta(r)^2}{r^2}} \\
            C_{2,r} =
                \lfrac{8}{3}\sn_\delta\pa[\big]{\lfrac{r}{2}}^2
                \pa{\Lambda\sn_\delta\pa[\big]{\lfrac{r}{2}}^2
                +2\max\set{\abs{\Delta}, \abs{\delta}}\sn_\delta\pa{r}}.
        \end{gather*}
    \end{theorem}

    The quantity $\widehat{\alpha}_r$ depends on the size of $r$, so the distortion produced by precomposing with the exponential map may not be uniformly bounded if the manifold has negative curvature and is not compact---as is the case with the hyperbolic space. On the other hand, we can give exact bounds for this deformation so, in practice, we can assume an upper bound on $r$ and then tune it as necessary. This approximation is common when proving convergence in manifolds that depend on the curvature (see for example \parencite{bonnabel2013stochastic,zhang2016riemannian,sato2019riemannian,tripuraneni2018averaging,ahn2020nesterov}).

    Finally, we showcase how to use these bounds to prove convergence of static and dynamic trivializations. This is a corollary from the previous results and the convergence of gradient descent in $\RR^n$.
}{%
\input{intro_chapter.tex}%
}

\section{Differential geometry: Conventions and notation}\label{sec:background}
    In this section, we stablish the notation used to refer to some recurrent objects, such as the distance function, the segment domain, mixed derivatives and pullbacks of connections. We also use it to recall some definitions from differential geometry that are not that commonly seen in the area of optimization, such as the definition of the Lie derivative and covariant derivative of tensors.

    \paragraph{The manifold}
    We will always work on a connected and complete Riemannian manifold $(M, \gm)$ of dimension $m \geq 2$ and at least $C^4$ regularity. We will denote the sectional curvature on the plane defined by two vectors $u, v \in T_pM$ as $\sec(u,v)$. When we write $\sec \leq \Delta$ for a constant $\Delta \in \RR$, we mean that for every $p \in M$, $u,v \in T_pM$, $\sec(u,v) \leq \Delta$. By a \emph{geodesic}, we will always mean a unit speed geodesic. We will implicitly identify $T_{\pa{p,v}}(T_pM) \iso T_pM$ for every $v \in T_pM$.

    \paragraph{Distance function, segment domain, conjugate and cut locus}
    For a point $p \in M$, we define the \emph{segment domain} as
    \[
        \segall \defi \set{v \in T_pM | \exp_p(tv) \text{ is length minimizing for }t \in [0, 1]}
    \]
    and we denote its interior by $\tseg$. The set $\tseg$ is a star-shaped open neighbourhood of $0$ in $T_pM$. On this neighbourhood the exponential map is a diffeomorphism. In particular, there exists an inverse $\exp_p^{-1}$ and $\dif\exp_p$ is full rank on $\tseg$. We will also write $\seg \defi \exp_p(\tseg)$. $\seg$ is sometimes referred as a \emph{normal neighborhood of $p$}, and its complement is called the \emph{cut locus} $\cut(p) \subset M$.

    A conjugate point $\exp_p(v) = q$ of $p$ is one such at which $\pa{\dif \exp_p}_v$ is not full rank. We denote the set of all these points $\conjloc(p)$. We have that $\conjloc(p) \subset \cut(p)$. By a theorem of Sard, $\conjloc(p)$ has measure zero in $M$~\parencite{sard1965hausdorff}. Even more, $\cut(p)$ has Hausdorff dimension at most $n-1$~\parencite{itoh1998dimension}. As such, $\seg$ is an open neighbourhood of $p$ that covers almost all the manifold.

    We will write $r(x) = d(x,p)$ for the distance to a fixed point $p$. This function is differentiable on $\seg \backslash \set{p}$, with gradient the unit radial vector field emanating from $p$. In particular, we have that for a geodesic $\gamma$ starting at $p$, $\grad r\vert_\gamma = \dgamma$. The cut locus of $p$ is exactly the set of points other than $p$ at which the distance function $r$ is not differentiable. This will be import

    \paragraph{Einstein convention}
    Whenever we refer to coordinates $\set{x^i}$ in $\seg$, we will always assume that these are the normal coordinates given by $\exp_p^{-1}$ and certain fixed frame on $T_pM$. We will denote $\partial_i \defi \frac{\partial}{\partial x^i}$ for short. We will use Einstein's summation convention that if an index appears as a super-index and a sub-index in a formula, it means that we are summing over it. For example, for a vector field $X$ in local coordinates we write
\[
    X = \sum_{i=1}^n X^i \partial_i = X^i \partial_i.
\]

    \begin{remark}
        Almost all the results in this paper can be developed working just on a neighbourhood of a geodesic. As such, most of the times it will not be necessary to work on a subset of $\seg$ but merely on a neighbourhood of a geodesic at which on which the exponential is a diffeomorphism. We will make sure of make this explicit in each result throughout the paper.
    \end{remark}

    \paragraph{Connections and derivations}
    We will write $D_X f \defi \dif f(X)$ for the directional derivative of a function along a vector field $X$ to disambiguate with the gradient of a function, which we will denote by $\grad f$ its gradient. We denote by $\conn$ the Levi-Civita connection on $\seg$ and by $\connflat$ the pushforward of the flat connection on $T_pM$ along $\exp_p$. In the same way that we do for the norms, we will abuse the notation and also denote by $\conn$ the associated connections defined by $\conn$ in the associated bundles. We recall that connections on tensor bundles are defined so that the Leibnitz rule holds. For example, for the Hessian and three vector fields $X, Y_1, Y_2$,
    \[
    D_X\pa{\Hess f\pa{Y_1,Y_2}} = \pa{\conn_X\Hess f}\pa{Y_1, Y_2} + \Hess f\pa{\conn_X Y_1, Y_2} + \Hess f\pa{Y_1, \conn_X Y_2}
    \]
    so we define
    \[
     \pa{\conn_X\Hess f}\pa{Y_1, Y_2} \defi
    D_X\pa{\Hess f\pa{Y_1,Y_2}} - \Hess f\pa{\conn_X Y_1, Y_2} - \Hess f\pa{Y_1, \conn_X Y_2}.
    \]

    We can do the same for a tensor $T$ of type $(0,s)$. We define its covariant derivative $\conn T$ as the tensor of type $(0, s+1)$ such that the Leibnitz rule holds
    \[
        \pa{\conn_X T}(X_1, \dots, X_s) \defi D_X\pa{T\pa{X_1, \dots X_s}} - \sum_{i=1}^s T\pa{X_1, \dots, \conn_X X_i, \dots, X_s}.
    \]
    For example, if we have in local coordinates the differential of a function is a $(0, 1)$ tensor $\dif f = f^i \dif x_i$. We may compute its Hessian in local coordinates as the $(0,2)$ tensor given by
    \[
        \conn \dif f = \conn(f^i\dif x_i) = \dif f^i \tensor \dif x_i + f^i \conn \dif x_i
    \]
    where we have used that the connection on functions is just the differential, by definition.

    The Lie derivative of tensors is defined in the same way. For a $(0, s)$ tensor $T$, the Lie derivative of $T$, $\lie_X T$ in the direction of a tensor $X$ is defined as the $(0, s)$ tensor
    \[
        \pa{\lie_X T}(X_1, \dots, X_s) \defi D_X\pa{T\pa{X_1, \dots X_s}} - \sum_{i=1}^s T\pa{X_1, \dots, \lie_X X_i, \dots, X_s}.
    \]

    \paragraph{Pullback connections}
    When dealing with a smooth curve $\deffun{\gamma : [0,r] -> M;}$, we will sometimes want to derive a vector field $X$ along it. We will write $\conn_{\partial_t}X$ for the \emph{covariant derivative along $\gamma$}, that is, the \emph{pullback connection along $\gamma$}.

    We will sometimes write the covariant derivative along $\gamma$ as $\dot{X} \defi \conn_{\partial_t} X$. This comes from the fact that if we choose a parallel frame $\set{e_i}$ along $\gamma$, that is $\conn_{\partial_t} e_i = 0$, we have that by the Leibnitz rule
    \[
        \dot{X} = \conn_{\partial_t}\pa{X^ie_i} = \conn_{\partial_t}\pa{X^i}e_i + X^i\conn_{\partial_t}\pa{e_i} = \dot{X}^ie_i
    \]
    where $\dot{X}^i$ are just the regular derivatives of the coordinate functions. With this notation, the usual equation for geodesics simply reads $\ddot{\gamma} = 0$.

    If instead of a curve we have an embedded surface
    \[
    \deffun{c : [0,r] \times [-\epsilon, \epsilon] -> M;
            (t,s) -> c(t,s)}
    \]
    we will analogously write $\partial_s$ for the vector field in the direction of the second component. Suppose that we have a vector field $J$ along $\gamma(t) \defi c(t,0)$ given by
    \[
        J(t) = \frac{\partial c}{\partial s}(t, 0) = \dif c(\partial_s) \vert_{\gamma}
    \]
    where $\partial_s$ is the coordinate vector field on the second component of $[0,r] \times [-\epsilon, \epsilon]$. We will abuse the notation and write for a vector field $X$ along $\gamma$
    \[
        \conn_J X \defi \conn_{\partial_s} X
    \]
    in the same way that we may write the equation for the geodesics as
    \[
        \conn_{\dgamma}\dgamma \defi \conn_{\partial_t}\dgamma = 0.
    \]

    Sometimes, we will also pullback a connection along the exponential map $\exp_p$. For the distance function $r$ to $p$, we have the radial vector field on $\seg \backslash \set{p} \subset M$ given by $\grad r$. By Gauss's lemma, the pullback of this vector field to $\tseg \backslash \set{0} \subset T_pM$ via the exponential map is exactly the radial vector field on the tangent space $\partial_r$, that is
    \[
        \dif \exp_p(\partial_r) = \grad r.
    \]
    With this notation, we can write the equation for the geodesics on the whole $\seg \backslash \set{p}$ using the pullback connection along $\exp_p$ as
    \[
        \conn_{\partial_r} \grad r = 0.
    \]

    \paragraph{Hessian and iterated Hessian}
    If we want to evaluate the $(0,2)$-Hessian twice on the same vector, we may do so by simplifying this problem to just performing a Euclidean derivative. Let $\gamma$ be the geodesic such that $\dgamma(0) = X$
    \begin{equation}\label{eq:hess_derivative_along_geodesic}
        \pa{f \circ \gamma}''(0) = D_{\dgamma(0)}\scalar{\grad f, \dgamma} = \scalar{\conn_{\dgamma(0)} \grad f, \dgamma(0)} = \conn\dif f(X,X) = \Hess f(X, X).
    \end{equation}
    We will use the notation $\conn \dif f$ to refer to the Hessian, given that it generalizes to smooth maps between manifolds such as the exponential map.

    We also recall the definition of the iterated Hessian of a function as a $(0,2)$ tensor. This is also best introduced in its $(1,1)$ form
    \[
        \Hess^2 f(X) \defi \pa{\Hess f \circ \Hess f}(X) = \conn_{\conn_X \grad f}\grad f.
    \]
    Its $(0, 2)$ version is consequently defined as
    \begin{equation}\label{eq:def_iterated_hessian}
        \Hess^2 f(X,Y) \defi \scalar{\Hess f(\Hess f(X)), Y} = \Hess f(\Hess f(X), Y) = \scalar{\Hess f(X), \Hess f(Y)}
    \end{equation}
    where we have used that the Hessian is symmetric. Through this formula it is clear that $\Hess^2 f$ is also symmetric.

    \section{First Order Bounds for the Exponential Map}\label{sec:first_order_bounds}
In this section, we give bounds on the norm of the differential of the exponential map of a manifold with bounded sectional curvature. In particular, if the sectional curvature of any plane of $M$ is bounded above and below by $\delta \leq \sec \leq \Delta$, for a ball $B_p(r)$, the bounds will be of the form
\[
    f_\Delta(r)\norm{w} \leq \norm{\pa{\dif \exp_p}_{v}(w)} \leq f_\delta(r)\norm{w} \mathrlap{\qquad \forall w \in T_pM, \norm{v} \leq r}
\]
for $r \geq 0$ and suitable functions $\deffun{f_\delta, f_\Delta : [0, r] -> \RR^+;}$.

First order bounds for the exponential map have been known since the times of Cartan~\parencite{cartan1928lesons}, and are well known in areas such as comparison geometry or PDEs, but they are not so known in the area of optimization on manifolds. The tight bounds that we will present here are often referred to as \emph{Rauch's theorem} as he was the first to show these bounds in the positive curvature-case~\parencite{rauch1951contribution} and later in the general case~\parencite{rauch1959geodesics}.

There are quite a few proof techniques of these theorems. There exist proofs via the second variation formula of the energy~\parencite{spivak1999comprehensive}, through bounds on the norm via the Jacobi equation~\parencite{jost2017riemannian}, by bounding Riccati equation on self-adjoint operators and integrating the result~\parencite{eschenburg1994comparison}, or by the study of the distance function to a given point~\parencite{gromov1981structures}. We choose to present here this fourth approach.

The approach presented here was first introduced by Gromov in the context of volume bounds, in what's called now the Bishop--Gromov theorem. This approach has the advantage of yielding upper and lower bounds for positive and negative curvatures at the same time. Other approaches need different techniques for the upper and lower bounds or they just give some weaker version of the theorem with constraints on the sign of the curvature. This approach also has a more geometric flavour as it accounts for bounding the principal curvatures of geodesic balls on the manifold. The general strategy of this proof is to decompose certain homogeneous second order differential equation into a Riccati equation and a system of first order equations, bound the Riccati equation, and then integrate the result to get the bounds on the differential of the exponential. This approximation to Jacobi fields and parallel vector fields is can be found in~\parencite{cheeger2008comparison}, and the approximation using curvature equations is from~\parencite{petersen2016riemannian}. Some of the proofs are either corrections or simplifications of the original ones.

    \subsection{Jacobi fields}
    We begin by introducing what will be the main tool that we will just throughout the paper: Jacobi fields. Jacobi fields are certain vector fields along a given geodesic that describe the behavior of the differential of the exponential map along this geodesic. We will see that they are the solution of a certain differential equation, and we will use this equation to give the bounds on the differential of the exponential.

    Before defining what Jacobi fields are and deducing this differential equation we will need a lemma. This lemma that can loosely be interpreted as \emph{the Lie derivative commutes with the differential of a smooth map}.

    \begin{lemma}\label{lemma:lie_derivative_pushforward}
        For a diffeomorphism $f$ and two vector fields $X, Y$ on $M$ we have that
        \[
            \dif f([X,Y]) = [\dif f(X), \dif f (Y)].
        \]
    \end{lemma}
    \begin{proof}
        See, for example,~\parencite[][1. Lemma 22]{oneill1966fundamental}.
    \end{proof}

    With this lemma in hand, we are ready to show that the differential of the exponential satisfies certain differential equation.

    \begin{proposition}[First order Jacobi equation]\label{prop:riccati_jacobi}
        Let $(M, \gm)$ be a complete Riemannian manifold. Let $\deffun{\gamma : [0, r] -> M;}$ be a geodesic in $\seg$ with initial unit vector $v \defi \dgamma(0)$ and a vector $w \in T_pM$ such that $w \perp v$. Consider the following variation of $\gamma$ in the direction of $w$
        \[
            \deffun{c : [0, r] \times (-\epsilon, \epsilon) -> M ; (t, s) -> \exp_p(t(v+sw))}
        \]
        where $\epsilon$ is small enough so that $c$ defines an embedded surface.

        Define the vector field $J$ along $\gamma$ as
        \[
            J(t) \defi \dif c(\partial_s)\vert_\gamma = \frac{\partial c}{\partial s}(t, 0) = \pa{\dif \exp_p}_{tv}(tw).
        \]
        This vector field solves the following differential equation on $\gamma$
        \[
            \lie_{\grad r} J \vert_{\gamma} = 0
        \]
        or equivalently, using the pullback connection
        \begin{equation}\label{eq:riccati}
            \conn_{\dgamma} J = \conn_J \grad r.
        \end{equation}
    \end{proposition}
    \begin{proof}
        The fact that $\gamma$ is in $\seg$ is equivalent to saying that $tv \in \tseg$ for $t \in [0, r]$. From this we get that the $\epsilon$ in the definition exists as $\tseg$ is open. Since $c$ defines an embedded surface, $c$ is a diffeomorphism onto its image.

        Since $\partial_t, \partial_s$ are coordinate vector-fields of the surface defined by $c$, we have that
        \[
            [\partial_t, \partial_s] = 0.
        \]
        Using that $c$ is a diffeomorphism, by~\Cref{lemma:lie_derivative_pushforward} together with the fact that the Lie derivative restricted to a submanifold is the Lie derivative of the restrictions, we have that
        \[
            [\dif c\pa{\partial_t},\dif c\pa{\partial_s}] = \dif c\pa{[\partial_t, \partial_s]} = 0.
        \]
        We get the Lie formulation of the Jacobi equation by noting that $\dif c(\partial_t) \vert_\gamma = \grad r \vert_\gamma$. The equation using the pullback connection is just a reformulation of the Lie derivative using that the Levi-Civita connection is torsion-free.
    \end{proof}

    \begin{remark}
        Note that, even though the differential equation involves a Lie bracket, we just need to define $J$ along the flow of $\grad r$, which is just a geodesic $\gamma$. This equation can be extended into a PDE on the whole $\seg \backslash \set{p}$, looking for vector fields on this domain such that
        \[
            \lie_{\grad r}W = 0.
        \]
        One such a solution is clearly given by the radial vector field $W = \grad r$.

        To find other solutions, we imitate the construction that we did in~\Cref{prop:riccati_jacobi} and define $W$ on a sphere around $p$ and extend it radially to the whole $\seg$ as
        \[
            W(x) = \dif\exp_p\pa[\Big]{r(x)W\pa[\Big]{\lfrac{x}{r\pa{x}}}}.
        \]
        It is clear by~\Cref{prop:riccati_jacobi} that this vector field solves the first order Jacobi equation on the whole $\seg \backslash \set{p}$.
    \end{remark}

        We will now use this equation to derive the linear version of the Jacobi equation.

    \begin{proposition}[Jacobi equation]\label{prop:jacobi_equation}
    Let $\deffun{\gamma : [0, r] -> M;}$ be a geodesic with initial values $\gamma(0) = p$, $\dot{\gamma}(0) = v$. Then $J(t) = \pa{\dif \exp_p}_{tv}(tw)$ is a vector field along $\gamma$ that solves the following second order homogeneous linear differential equation
    \[
        \ddot{J} + R(J, \dgamma)\dgamma = 0.
    \]
    We call this equation the \emph{Jacobi equation}.
\end{proposition}
\begin{proof}
    If $w$ is parallel to $v$, $J$ takes the form $J(t) = t\norm{w}\dgamma(t)$ or $J(t) = -t\norm\dgamma(t)$. These fields solve the differential equation, since $R(\dgamma, \dgamma)\dgamma = 0$.

    If $w$ is normal to $v$, we may differentiate~\eqref{eq:riccati} to get
    \[
        \conn_{\partial_t} \conn_{\partial_t} J = \conn_{\partial_t} \conn_J \grad r
    \]
    and using the definition of the curvature tensor and the fact that $[J, \grad r] = 0$ we get
    \[
        \conn_{\partial_t} \conn_{\partial_t} J + R(J, \grad r)\grad r = 0.
    \]
    We then get the Jacobi equation by restricting this equation to $\gamma$, since $\grad r \circ \gamma = \dgamma$.

    Using an orthonormal parallel frame $\set{e_i}$ along $\gamma$ such that $e_1 = \dgamma$, and expressing a solution of this equation as $J^ie_i$, where $\deffun{J^i : [0, r] -> \RR;}$, we see that, in fact, this is a second order linear differential equation on $[0, r]$
    \[
        \ddot{J}^i + R_j^iJ^j = 0 \mathrlap{\qquad i = 1, \dots, n}
    \]
    with coefficients
    \[
        R_j^i =
        \begin{cases}
            0 \qquad &\text{if } i = 1 \\
            R(e_j, \dgamma, \dgamma, e_i) \qquad &\text{if } i = 2, \dots, n.
        \end{cases}\qedhere
    \]
\end{proof}

\begin{definition}[Jacobi field]
    We say that a vector field $J$ along a geodesic $\gamma$ is a \textbf{Jacobi field} if it satisfies the Jacobi equation.
\end{definition}

\begin{remark}[Basic properties of Jacobi fields]
    By elementary theory of differential equations, the Jacobi equation has $2n$ independent solutions, defined by the initial values $(J(0), \dot{J}(0)) \in T_pM \times T_pM$. By construction, we have found $n$ independent solutions given by
    \[
        J(t) = \pa{\dif \exp_p}_{tv}(tw) \mathrlap{\qquad \forall w \in T_pM.}
    \]
    These correspond to the initial values $J(0) = 0$, $\dot{J}(0) = w$. Geometrically, they correspond to a family of geodesics $c$ that fixes the initial point, that is, $c(0, s) = p$.

    It is also direct to show that $\dgamma$ is another solution to this equation with initial values $J(0) = v$, $\dot{J}(0) = 0$. The other $n-1$ independent solutions correspond to variations of the geodesic $\gamma$ that do not fix the initial point $p$. These solutions will not be important for our analysis.

    It starts being clear now the importance of Jacobi fields. If we want to control the norm of the differential of $\exp_p$ at a point $rv \in \tseg$, for $r > 0$, $\norm{v} = 1$, we may define the Jacobi field along the geodesic $\deffun{\gamma : [0,r] -> M;}$ with initial conditions $(p,v)$ and we have that
    \[
        \pa{\dif\exp_p}_{rv}(w) = \frac{J(r)}{r}.
    \]
    In order to bound the norm of the differential of the exponential, we just need to bound the norm of the solutions of the Jacobi equation.

    If the vector $w$ is parallel to $v$, as we saw in the proof of~\Cref{prop:jacobi_equation}, the Jacobi field takes the form $J(t) = t\norm{w}\dgamma(t)$. As such, in this direction we have an exact solution of the Jacobi equation and we can compute its norm exactly as $\norm{J(t)} = t\norm{w}$.

    If the vector $w$ is perpendicular to $\dgamma(0)$, then $J(t)$ is perpendicular to $\dgamma(t)$ for every $t \in [0, r]$, as the equation for $J^1$ would be given by
    \[
        \ddot{J}^1 = 0 \mathrlap{\qquad (J^1(0), \dot{J}^1(0)) = (0,0).}
    \]

    Geometrically, the previous proposition says that Jacobi fields just rotate around the vector field defined by $\dgamma$. In symbols, this means that if we split a Jacobi field along $\gamma$ into its radial and normal part as
    \[
        \paral{J} = \scalar{J, \dgamma}\dgamma \qquad \normal{J} = J - \paral{J}
    \]
    if $\paral{\dot{J}}(0) = 0$ then $\paral{J}(t) = 0$ for every $t \in [0, r]$.

    The results discussed in this remark are commonly known as the \emph{Gauss's lemma}.

    Jacobi fields are also very closely related to conjugate points and the conjugate locus. A point $q = \exp_p(v)$ is conjugate to $p$ if the exponential at $v$ is not full rank. Another way of looking at this definition is by defining a point $q$ conjugate to $p$ if there exists a Jacobi field connecting $p$ and $q$ such that it is zero at $p$ and $q$. We will see when studying Rauch's theorem that the estimates for Jacobi fields will work on every point on the manifold but the conjugate locus, as these points will be exactly the singularities of the maps that we work with. For a unit vector $v$, we will denote by $r_{\conjloc}(v) \in (0, \infty]$ the smallest number at which $q = \exp_p(r_{\conjloc}v)$ is conjugate to $p$.
\end{remark}

    \subsection{Parallel vector fields}
    In the same way that Jacobi fields with initial condition $J(0) = 0$ and $\dot{J}(0) \perp \dgamma(0)$ are the vector fields such that $\lie_{\grad r} J = 0$, we have parallel vector fields.

    \begin{definition}
        We say that $E$ is a parallel vector field on $\seg \backslash \set{p}$ if it satisfies
        \[
            \conn_{\grad r} E = 0.
        \]
    \end{definition}

    Parallel vector fields can be easily described along a geodesic as parallel transporting a vector $w \in T_{\gamma(0)}M$ along it. Analogously, we can also define a parallel vector field on the whole $\seg \backslash \set{p}$ by specifying a vector field on a sphere $S_\epsilon = \set{x \in M | d(x, p) = \epsilon} \subset\seg$ and extending it to the whole $\seg$ parallel transporting this vector field along geodesics.

    \subsection{Rauch's theorem}
    We now go back to the study of the norm of the differential of the exponential. The strategy that we will follow will be that of bounding the derivative of the norm of the exponential. For that end, let $J$ be a Jacobi field along a geodesic $\gamma$ such that $J(0) = 0$, $\dot{J}(0) \perp \dgamma(0)$. As we already saw before, this vector field takes the form
    \[
        J(t) = \pa{\dif \exp_p}_{t\dgamma(0)}\pa{t\dot{J}(0)}.
    \]
    The derivative of its norm is given by
    \[
        \frac{\dif}{\dif t}\norm{J} = \frac{\scalar{\dot{J}, J}}{\norm{J}}.
    \]
    At first sight it looks like we have not achieved much, as we still have a term involving the square of the norm of $J$ on the right hand side, but it turns out that it can be rewritten in a convenient way using~\eqref{eq:riccati},
    \[
        \scalar{\dot{J}, J} = \scalar{\conn_{\grad r} J, J} = \scalar{\conn_J \grad r, J} = \Hess r(J, J).
    \]

    Our plan will be to bound the Hessian of the distance to $p$ on vectors of norm $1$ to get bounds on the log-derivative of the norm of the exponential map and then integrate these to get bounds on $\norm{J}$. To this end, we start by computing formulas that relate Hessian of the distance function to the curvature tensor.

    A geometric interpretation of this approach comes after noting that the Hessian of the distance function is exactly the second fundamental form (or shape operator) of the distance function. As such, this quantity can be interpreted geometrically as the variation of the curvature of geodesic spheres $S_t(p) = \set{x \in M | d(x,p) = t}$ in the radial direction.

    \begin{proposition}[Radial Curvature Equations]\label{prop:radial_curvature}
        Let $(M, \gm)$ be a Riemannian manifold. For a point $p \in M$, we have on $\seg$
        \begin{align}
            \pa{\lie_{\grad r}\Hess r}(X, Y)& - \Hess^2 r(X, Y) = - R(X, \grad r, \grad r, Y)\label{eq:curv1} \\
            \pa{\conn_{\grad r}\Hess r}(X, Y)& + \Hess^2 r(X, Y) = - R(X, \grad r, \grad r, Y)\label{eq:curv2}.
        \end{align}
    \end{proposition}
        where $\Hess^2$ denotes the \emph{iterated Hessian} defined in~\eqref{eq:def_iterated_hessian}.
    \begin{proof}
        Since for a geodesic $\gamma$ we have that $\grad r \vert_\gamma = \dgamma$, the gradient of $r$ has $\conn_{\grad r} \grad r = 0$. From this we get
        \begin{align*}
            - R(X, \grad r, \grad r, Y) &= \scalar{\conn_{\grad r} \conn_X \grad r, Y} + \scalar{\conn_{\lie_X \grad r} \grad r, Y}\\
                                        &= D_{\grad r} \scalar{\conn_X \grad r, Y} - \scalar{\conn_X \grad r, \conn_{\grad r} Y} + \Hess r(\lie_X \grad r, Y)\\
                                        &= D_{\grad r} \pa{\Hess r(X, Y)} - \Hess r\pa{X, \conn_{\grad r} Y}
                                        - \Hess r(\lie_{\grad r} X, Y).
        \end{align*}
        Equation \eqref{eq:curv1} follows after expanding the first term as
        \[
            D_{\grad r} \pa{\Hess r(X,Y)} = \pa{\lie_{\grad r}\Hess r}(X, Y) + \Hess r\pa{\lie_{\grad r}X, Y} + \Hess r\pa{X, \conn_{\grad r} Y - \conn_Y \grad r}
        \]
        and equation \eqref{eq:curv2} follows after expanding it as
        \[
            D_{\grad r} \pa{\Hess r(X,Y)} = \pa{\conn_{\grad r}\Hess r}(X, Y) + \Hess r\pa{\conn_{\grad r}X, Y}  + \Hess r\pa{X, \conn_{\grad r} Y}.\qedhere
        \]
    \end{proof}

    \begin{remark}[A Riccati-type equation]
        Note that the previous proposition holds for every vector field $Y$. As such, we may write it in its $(1,1)$ form
        \[
            \pa{\conn_{\grad r} \Hess r}(X) + \Hess^2 r(X) + R(X, \grad r)\grad r = 0.
        \]
        Denoting by $R_2(X) \defi R(X, \grad r)\grad r$ the $(1,1)$ curvature tensor in this equation, and the shape operator for the geodesic balls as $S \defi \Hess r$, this equation can be rewritten in the radial direction as a Riccati equation on symmetric $(1,1)$ tensors
        \[
            S' + S^2 + R_2 = 0.
        \]
        When restricted to a geodesic, this can be seen as a differential equation on matrices, where $R_2$ is self-adjoint with respect to $\gm$. As one does in one dimension splitting a second order differential equation into a first order equation and a Riccati equation, here we have split the Jacobi equation into first order matrix equation---more formally, an equation on symmetric endomorphisms along the tangent spaces of $\gamma$---and a first order equation for a parallel vector field $E$ of the form
        \[
            \dot{E} = S(E).
        \]
        A careful analysis of this Riccati equation yields another particularly clean proof of Rauch's theorem, at the expense of the use of more abstract methods, as presented in~\parencite{eschenburg1990comparison}.

        In contrast, we will use parallel and Jacobi vector fields to simplify these matrix equations.
    \end{remark}

    \begin{proposition}\label{prop:radial_curvature_evaluated}
        Let $\gamma$ be a geodesic, and let $J$ be a Jacobi field along it such that $J(0) = 0$, $\dot{J} \perp \dgamma(0)$. Let $E$ be a parallel vector field along $\gamma$ that is normal to $\dgamma$. We have the following differential equations along $\gamma$
        \begin{align}
            \frac{\dif}{\dif t}\pa{\Hess r(J, J)}& - \Hess^2 r(J, J) = - \sec(J, \dgamma)\scalar{J, J}\label{eq:curv1_jacobi} \\
            \frac{\dif}{\dif t}\pa{\Hess r(E, E)}& + \Hess^2 r(E, E) = - \sec(E, \dgamma)\label{eq:curv2_parallel}.
        \end{align}
    \end{proposition}
    \begin{proof}
        We just evaluate~\eqref{eq:curv1} and~\eqref{eq:curv2} on $J, E$ and use that
        \[
            \lie_{\grad r} J = 0 \qquad \conn_{\grad r} E = 0.
        \]
        To go from the Riemannian tensor to the sectional curvature we just recall that if a Jacobi field or a parallel vector field are normal to $\dgamma$ at one point they are normal to $\dgamma$ at every point.

        Along a geodesic we have that the radial vector field $\grad r$ is just $\dgamma$, and as such we have that $D_{\grad r}\vert_\gamma = D_{\dgamma} = \frac{\dif}{\dif t}$.
    \end{proof}

    It is now a bit clearer how Jacobi fields may be used to give lower bounds and parallel vector fields to give upper bounds on the Hessian of the distance function.

    We will now define some generalized trigonometric functions, which will be useful in the sequel.
    \begin{definition}[Generalized trigonometric functions]
        For a constant  $\kappa \in \RR$, we define the generalized sine function $\sn_\kappa$ as the solution to the differential equation
        \[
            \ddot{x} + \kappa x = 0 \qquad x(0) = 0,\,\dot{x}(0) = 1.
        \]
        In particular, we have
        \[
            \sn_\kappa(t) \defi
            \begin{cases}
                \frac{\sin (\sqrt{\kappa}t)}{\sqrt{\kappa}}    \qquad &\text{if } \kappa > 0\\
                t                                              \qquad &\text{if } \kappa = 0\\
                \frac{\sinh (\sqrt{-\kappa}t)}{\sqrt{-\kappa}} \qquad &\text{if } \kappa < 0
            \end{cases}
        \]
        with its first positive zero being at
        \[
            \pi_\kappa \defi
            \begin{cases}
                \frac{\pi}{\sqrt{\kappa}}  &\qquad \text{if } \kappa > 0 \\
                \infty                     &\qquad \text{if } \kappa \leq 0
            \end{cases}
        \]

        We define the generalized cotangent $\ct_\kappa$ as
        \[
            \ct_\kappa(t)
            \defi
            \frac{\sn'_\kappa(t)}{\sn_\kappa(t)}
            =
            \begin{cases}
              \sqrt{\kappa}\cot(\sqrt{\kappa}t)    &\mathrlap{\qquad \text{if } \kappa > 0}\\
              \frac{1}{t}                          &\mathrlap{\qquad \text{if } \kappa = 0}\\
              \sqrt{-\kappa}\coth(\sqrt{-\kappa}t) &\mathrlap{\qquad \text{if } \kappa < 0}
            \end{cases}
        \]
        This function is smooth on $(0, \pi_\kappa)$.
    \end{definition}

    \begin{remark}
        Note that the equation $\ddot{x} + \kappa x = 0$ is the associated second order equation to the Riccati equation $\dot{x} + x^2 + \kappa = 0$. This is exactly the Riccati equation that the shape operator for the distance function solves in constant curvature, as pointed out before. Given that $\sn_\kappa$ solves the second order equation, $\ct_\kappa$ solves the Riccati equation.

        From an analytical point of view, it is worth noting that the family of functions $\sn_\kappa(t)$ is strictly increasing in $\kappa$ for $t \in (0, \pi_\kappa)$. We also have that $\sn_\kappa$ has a zero at $0$ for every $\kappa$ and another one at $\pi_\kappa$ for $\kappa > 0$, so
        \begin{align*}
            \lim_{t \to 0}\ct_\kappa(t) &=\infty\phantom{-}\mathrlap{\qquad\kappa \in \RR} \\
            \lim_{t \to \pi_\kappa} \ct_\kappa(t) &= -\infty \mathrlap{\qquad\kappa > 0.}
        \end{align*}
    \end{remark}

    All we need to do now is to relate the solutions of the Jacobi equation in Riccati as given by the Hessian of the distance function in~\Cref{prop:radial_curvature_evaluated} with its solutions in constant curvature. To do so, we need the following elementary comparison lemma for functions of real variable. It can be thought of as a version of Sturm's comparison theorem for Riccati equations.
    \begin{lemma}[Comparison Lemma for Riccati Equations]\label{lemma:comparison}
        Let $\deffun{\rho_1, \rho_2 : (a,b) -> \RR;}$ be two differentiable functions such that
        \[
            \dot{\rho}_1 + \rho_1^2 \leq \dot{\rho}_2 + \rho_2^2.
        \]
        If we have that $\rho_1(t_0) \geq \rho_2(t_0)$ for a $t_0 \in (a,b)$, then $\rho_1 \geq \rho_2$ on $(a, t_0]$.
    \end{lemma}
    \begin{proof}
        Let $u = \rho_1 - \rho_2$ and $s = \dot{\rho}_2 - \dot{\rho}_1 + \rho_2^2 - \rho_1^2$. We have by hypothesis that $u(t_0) \geq 0$ and $s \geq 0$ on $(a,b)$, and it is enough to prove that $u \geq 0$ on $(a, t_0]$. We can write the differential inequality as a differential equation in terms of $u$ and $s$ as
        \[
            \dot{u} = -(\rho_1 + \rho_2)u -s.
        \]
        and we can solve this differential equation using the method of variation of parameters. We compute the general solution of the homogeneous system
        \[
            \dot{v} = -(\rho_1 + \rho_2)v
        \]
        as $v = Cv_1$ for $v_1 \defi e^{-\int \rho_1 + \rho_2} > 0$ and a constant $C \in \RR$.

        To get a particular solution to the inhomogeneous equation, we let $u = \eta v_1$ and, plugging it into the inhomogeneous equation, we see that the function $\eta$ satisfies $\dot{\eta} = -sv_1^{-1}$. As $s$ and $v_1^{-1}$ are positive, we conclude that $\eta$ is decreasing.

        Finally, we get the general solution by adding the particular solution and the general solution to the homogeneous system $u = v_1(C + \eta)$. Since, by hypothesis, $u(t_0) \geq 0$, we get that $C + \eta(t) \geq 0$ for $t \in (a, t_0]$ so that $u(t) \geq 0$ on $(a, t_0]$.
    \end{proof}
    \begin{proposition}[Riccati Comparison Estimate]\label{prop:riccati_comparison_estimate}
        Let $\deffun{\rho : (0, b) -> \RR;}$ be a differentiable function and fix a $\kappa \in \RR$.
        \begin{enumerate}
        \item
            If $\dot{\rho} + \rho^2 \leq -\kappa$ then $\rho(t) \leq \ct_\kappa(t)$ on $(0,b)$. Furthermore, $b \leq \pi_\kappa$

        \item
            If $\dot{\rho} + \rho^2 \geq -\kappa$ and $\lim\limits_{t \to 0}\rho(t) = \infty$ then $\ct_\kappa(t) \leq \rho(t)$ on $(0, \min\set{b, \pi_\kappa})$.
        \end{enumerate}
    \end{proposition}
    \begin{proof}
        For the first case, assume that $\rho\pa{t_0} > \ct_\kappa(t_0)$ for some $t_0 \in (0, b)$. By continuity, we can choose an $\epsilon > 0$ such that
        \[
            \rho\pa{t_0} \geq \ct_\kappa(t_0-\epsilon).
        \]
        By~\Cref{lemma:comparison}, since $\ct_\kappa$ solves the differential equation $\dot{x} + x^2 + \kappa = 0$, we have that this inequality holds for $t \in (\epsilon, t_0)$, but the right hand side goes to infinity as $t \to \epsilon$. In this case we also have that $b \leq \pi_\kappa$, as the comparison function goes to $-\infty$ as $t$ goes to $\pi_\kappa$.

        Interchanging the roles of $\rho$ and $\ct_\kappa$ in the previous proof we get the second inequality. In this case we need $t \in (0, \min\pa{b, \pi_\kappa})$ for $\ct_\kappa$ to be well-defined as $\ct_\kappa \to -\infty$ when $t$ goes to $\pi_\kappa$ for $\kappa > 0$.
    \end{proof}

    We now have everything we need to prove Rauch's theorem.

    \begin{theorem}[Rauch's theorem]\label{thm:rauch}
        Let $(M, \gm)$ be a Riemannian manifold with bounded sectional curvature $\delta \leq \sec \leq \Delta$. Fix a point $p \in M$ and define $r$ as the distance to $p$. For any other point in a ball $x \in B_p(\pi_\Delta) \backslash \set{p}$, and any $w_1, w_2 \in T_xM$ we have
        \[
            \ct_\Delta(r(x))\scalar{\normal{w}_1, \normal{w}_2}
            \leq
            \pa{\Hess r}_x(w_1, w_2)
            \leq
            \ct_\delta(r(x))\scalar{\normal{w}_1, \normal{w}_2}.
        \]
        The upper bound also holds for every $x \in M \backslash \pa{\cut(p) \cup \set{p}}$. These bounds are tight on spaces of constant curvature.
    \end{theorem}
    \begin{proof}
        First, note that since $\Hess r \vert_x$ is a symmetric bilinear form, it is diagonalizable and it will achieve its maximum and minimum at certain eigenvectors. As such, we just have to prove the result for $w_1 = w_2$.

        Consider a distance minimizing geodesic $\deffun{\gamma :[0,r] -> M;}$ connecting $p$ and $x$. For a radial vector $w = \norm{w}\dgamma(r)\in T_xM$ we evaluate the Hessian on the vector field $\norm{w}\dgamma$ using that $\grad r\vert_\gamma = \dgamma$
        \[
            \Hess r(\norm{w}\dgamma, \norm{w}\dgamma) = \norm{w}^2 \scalar{\conn_{\dgamma} \dgamma, \dgamma} = 0.
        \]
        So the Hessian of the distance function is zero in the radial direction.

        By linearity, it is enough to prove the theorem for a unitary vector $w$ normal to $\dgamma(r)$.

        For the rest of the bounds, we will take advantage of the simple form that the curvature equations take when evaluated on Jacobi and parallel vector fields, as computed in~\Cref{prop:radial_curvature_evaluated}.

        For the lower bound we consider a Jacobi field such that $J(r) = w$ (\ie, $J(0) = 0$ and $\dot{J}(0) = \frac{w}{r}$) and define $\rho(t) = \Hess r(\frac{J}{\norm{J}}, \frac{J}{\norm{J}})$. We can compute its derivative as
        \[
            \dot{\rho}
            = \frac{\dif}{\dif t} \frac{\Hess r(J, J)}{\scalar{J, J}}
            = \frac{\frac{\dif}{\dif t}\Hess r(J, J) \norm{J}^2 - 2\Hess r(J, J)^2}{\norm{J}^4}.
        \]
        Plugging the formula for the derivative of the Hessian into~\eqref{eq:curv1_jacobi} we get
        \[
            \dot{\rho} + 2\rho^2 - \Hess^2 r\pa[\Big]{\frac{J}{\norm{J}},\frac{J}{\norm{J}}} = -\sec(J, \dgamma).
        \]
        Using Cauchy-Schwarz, we can bound the iterated Hessian for any vector field $X$ of norm $1$
        \[
            \Hess^2 r(X,X) = \scalar{\conn_X\grad r, \conn_X\grad r} \geq \scalar{\conn_X \grad r, X}^2 = \Hess r(X,X)^2
        \]
        and together with the upper bound on the sectional curvature we get the expected differential inequality
        \[
            \dot{\rho} + \rho^2 \geq - \Delta.
        \]
        Finally, since $\norm{J(0)} = 0$ and $\norm{\dot{J}(0)} \neq 0$, we have that $\lim\limits_{t \to 0} \rho(t) = \infty$, so we may use the second part of~\Cref{prop:riccati_comparison_estimate} to finish the lower bound.

        For the upper bound consider a parallel vector field $E$ such that $E(r) = w$ and define $\rho = \Hess r(E, E)$. Using~\eqref{eq:curv2_parallel} we get that
        \[
            \dot{\rho} + \rho^2 = -\sec(E, \dgamma) \leq -\delta
        \]
        and we finish by applying the first part of~\Cref{prop:riccati_comparison_estimate}.
    \end{proof}

    \begin{remark}
        For the upper bound, we never used the fact that $x \not\in \cut(p)$. We merely used that the geodesic $\gamma$ does not have any conjugate points on $[0,r]$, to be able to use~\Cref{prop:riccati_comparison_estimate}, as if $\norm{J(t)} = 0$, then the Hessian of the distance function at that point would be infinite.

	As stated, the theorem is as general as it can be, as the distance function $r$ is not differentiable on $\cut(p)$. On the other hand, if it is stated in terms of Jacobi fields, one can further generalize it to Jacobi fields along geodesics without conjugate points, since both the proof of the theorem and that of~\Cref{prop:radial_curvature} can be done in terms of a geodesic and short variations of geodesics. We will use this more general version in the following theorem.
    \end{remark}

    \begin{theorem}[First order bounds for the exponential map]\label{thm:first_order_bounds}
        Let $(M, \gm)$ be a Riemannian manifold with bounded sectional curvature $\delta \leq \sec \leq \Delta$. Fix a point $p \in M$, for any unit vector $v\in T_pM$ and an $r \in [0, \pi_\Delta]$,
        \[
            \min\set[\Big]{1, \frac{\sn_\Delta(r)}{r}}\norm{w}
            \leq
            \norm{\pa{\dif \exp_p}_{rv}(w)}
            \leq
            \max\set[\Big]{1, \frac{\sn_\delta(r)}{r}}\norm{w} \mathrlap{\qquad \forall w \in T_pM.}
        \]
        The upper bound also holds for $r < r_{\conjloc}(v)$.

        These bounds are tight on spaces of constant curvature.
    \end{theorem}
    \begin{proof}
        The bound by $1$ above and below comes from the Gauss's lemma, as the exponential is a radial isometry. From this, and using linearity, we just have to prove the bound for a vector $w$ in the normal direction with $\norm{w} = 1$.

        All we have to do now is to transform the bounds on the Riccati equation on bounds on the Jacobi equation. Define $\gamma$ as the geodesic starting at $p$ with $\dgamma(0) = \frac{v}{\norm{v}}$, and choose a Jacobi field along $\gamma$ such that $J(r) = w$, that is, choose $\dot{J}(0) = \frac{w}{r}$. As we noted at the beginning of this section, we have that
        \[
        \Hess r\pa[\Big]{\frac{J}{\norm{J}}, \frac{J}{\norm{J}}}
        = \frac{\scalar{\dot{J}, J}}{\norm{J}^2}
        = \frac{\lfrac{\dif}{\dif t}\norm{J}}{\norm{J}}
        = \frac{\dif}{\dif t}\log\pa{\norm{J}}
        \]
        so Rauch's theorem may be rewritten on $(0, \pi_\Delta)$ as
        \[
            \frac{\dif}{\dif t}\log\pa{\sn_\Delta(t)}
            \leq
            \frac{\dif}{\dif t}\log\pa{\norm{J}}
            \leq
            \frac{\dif}{\dif t}\log\pa{\sn_\delta(t)}.
        \]
        Integrating, using that $\norm{J(0)} = \sn_{\Delta}(0) = \sn_{\delta}(0) = 0$, and computing the limit using l'Hopital, we get
		\[
            \sn_\Delta(t) \leq \norm{J} \leq \sn_\delta(t).
		\]
		and since $J(r) = \pa{\dif \exp_p}_{rv}(rw)$,
		\[
            \norm{\pa{\dif \exp_p}_{rv}(w)} = \frac{\norm{J(r)}}{r}.
		\]

        As it was the case in the proof of Rauch's theorem, we just used that $\gamma$ does not have conjugate points on $[0,r]$, rather than the stronger $rv \in \tseg$. As such, just by noting that the proof does not rely on the definition of $\Hess r$, and instead, can be carried out in terms of Jacobi fields, we get the finer result for the upper bound.
    \end{proof}

    \begin{example}
        To see how the definition of the upper bounds are an improvement over just considering $rv \in \tseg$, consider the flat torus. For this manifold, $\tseg \subset T_pM$ is a square centered at $0$. On the other hand, since for the flat torus we have that $\conjloc(p) = \emptyset$, we get that the upper bound holds on all $T_pM$ for every point $p$.
    \end{example}

    We now have as a corollary of the lower bounds the following strengthening of the Cartan-Hadarmard theorem that was implicitly present in the statement of the last theorem.
    \begin{theorem}\label{thm:generalization_cartan_hadamard}
        If $(M, \gm)$ has $\sec \leq \Delta$, then the exponential map $\exp_p$ is not singular on $B_p\pa{ \pi_\Delta}$, so $B_p\pa{\pi_\Delta} \subset M \backslash \conjloc(p)$ and $\pi_\Delta \leq r_{\conjloc}(v)$ for every point $p \in M$ and every unit vector $v \in T_pM$.
    \end{theorem}

    \subsection{The law of cosines on manifolds}\label{sec:law_of_cosines}
    Rauch's theorem has found countless applications in geometry, as those given by Rauch himself, who proved with it a weak version of the sphere theorem, or as used by Gromov or Karcher to give bounds the volume form of a Riemannian manifold, or to control de size of the fundamental group of a given manifold. Here we deviate slightly from the presentation to remark one that has been found particularly useful in the context of optimization on manifolds.

    \begin{theorem}[Law of cosines]\label{thm:law_of_cosines}
        Let $(M, \gm)$ be a Riemannian manifold with bounded sectional curvature $\delta \leq \sec \leq \Delta$. Let $p$ be a point on $M$ and let $x, y \in B_p\pa{R}$ for $R \leq \pi_\Delta$ such that the minimizing length geodesic $\gamma$ that connects them also lies in this ball. Define the angle $\alpha = \angle pxy$. We then have
        \begin{align*}
            d(y, p)^2 &\leq \zeta_{1, \delta}(R)d(x, y)^2 + d(x, p)^2 - 2d(x, y)d(x, p)\cos(\alpha) \\
            d(y, p)^2 &\geq \zeta_{2, \Delta}(R)d(x, y)^2 + d(x, p)^2 - 2d(x, y)d(x, p)\cos(\alpha)
        \end{align*}
        where
        \begin{alignat*}{2}
            \zeta_{1,\kappa}(r)
            &\defi \max\set{1, r\ct_\kappa(r)}
            &= \begin{cases}
                1                                    &\text{if } \kappa \geq 0 \\
                \sqrt{-\kappa}r\coth(\sqrt{-\kappa}r)&\text{if } \kappa < 0
            \end{cases}
            \\
            \zeta_{2,\kappa}(r)
            &\defi \min\set{1, r\ct_\kappa(r)}
            &= \begin{cases}
                \mathrlap1\hphantom{\sqrt{-\kappa}r\coth(\sqrt{-\kappa}r)}  &\text{if } \kappa \leq 0 \\
                \sqrt{\kappa}r\cot(\sqrt{\kappa}r)                          &\text{if } \kappa > 0
            \end{cases}
        \end{alignat*}
        Furthermore, the first bound holds as long as there exists a distance minimizing geodesic contained in $\seg$ that connects $x$ and $y$. These bounds are tight on spaces of constant curvature.
    \end{theorem}
    \begin{proof}
         Define $f_p = \frac{1}{2}r^2$. We start by relating the Hessian of $f_p$ to the Hessian of $r$ using the chain rule for the differential and the Leibnitz rule for the connection
        \[
            \Hess f_p = \conn\pa{r \dif r} = \dif r \tensor \dif r + r\Hess r.
        \]
        As such, for a radial vector $w$ of norm $1$, since $\Hess r(w,w) = 0$, this equation simplifies to
        \[
            \Hess f_p(w,w) = \dif r(w)^2 = \scalar{\grad r, w}^2 = 1.
        \]
        For normal vectors we may use the bounds on the Hessian of $r$. It is now evident that the functions $\zeta_{1, \delta}$ and $\zeta_{2, \Delta}$ are defined to be upper and lower bound on the Hessian of $f_p$ acting on vectors of norm $1$.

        Let $\deffun{\gamma : [0, t] -> M;}$ be a distance-minimizing geodesic such that $\gamma(0) = x$, $\gamma(t) = y$ and $\gamma(s) \in \seg$ for every $s \in [0, t]$. We will prove the first inequality, as the proof of the second one is analogous. By the bounds above, using~\eqref{eq:hess_derivative_along_geodesic}, we have that
        \[
            \Hess f_p(\dgamma(s), \dgamma(s)) = (f_p \circ \gamma)''(s) \leq \zeta_{1, \delta}(R)
        \]
        for every $s \in [0, t]$, since $f_p$ is differentiable on $\seg$. Integrating this inequality we get
        \[
            \intf[0][t]{(f_p \circ \gamma)''(s)}{s} = (f_p \circ \gamma)'(t) - \scalar{\grad f_p(\gamma(0)), \dgamma(0)} \leq \zeta_{1, \delta}(R)t.
        \]
        Integrating once again we have that
        \[
            (f_p \circ \gamma)(t) - (f_p \circ \gamma)(0) - \scalar{\grad f_p(\gamma(0)), \dgamma(0)}t \leq \zeta_{1, \delta}(R)\frac{t^2}{2}.
        \]
        which, after substituting the values of $f_p$ and $\gamma$, gives
        \[
            d(y, p)^2 \leq \zeta_{1, \delta}(R)d(x,y)^2 + d(x,p)^2 + 2d(x,y)\scalar{\grad f_p(x), \dgamma(0)}.
        \]
        Finally, since $\grad f_p(x) = r \grad r = -\exp_x^{-1}(p)$ we get the desired inequality.
    \end{proof}

    \begin{remark}[Obstructions to a global law of cosines]
        One could ask whether it is possible to extend this local result on $\seg$ to a global result for geodesic triangles whose sides are distance minimizing, as one does in Toponogov's theorem. This is, in general, not possible. Consider the flat torus and a point $p$ in it and an equilateral triangle on a generating circle. The sides of this triangle are distance minimizing of length $x$---one third the circumference of the circle---and the angles of this triangle will be $\alpha = \pi$. Since for the flat torus we have that $\delta = 0$, the first inequality would imply that $3x^2 \leq 2x^2$.
    \end{remark}

    The upper bound for the case $\delta < 0$ has found multiples uses in the optimization literature. Before starting, to be able to put what comes next in context, we recall the definition of Hadamard manifolds.
    \begin{definition}[Hadamard manifold]
        A Riemannian manifold is a \textbf{Hadamard manifold} if it is complete, simply connected, and has non-positive sectional.
    \end{definition}

    These manifolds are called Hadamard manifolds due to the Cartan-Hadamard theorem.
    \begin{theorem}[Cartan-Hadamard]
        A Hadamard manifold is diffeomorphic to $\RR^n$.
    \end{theorem}
    \begin{proof}
        By~\Cref{thm:generalization_cartan_hadamard}, since $\sec \leq 0$, $\pi_\Delta = \infty$ and the differential exponential map at a point $p \in M$ is always full rank. As such, by the inverse function theorem, it is a covering map, and $\RR^n$ is the universal cover of $M$. Since $M$ is simply connected, $M$ is diffeomorphic to its universal cover.
    \end{proof}

    It is now clear that Hadamard manifolds are a class of particularly well-behaved manifolds, in that the $\cut(p) = \emptyset$ for every $p \in M$, so $\seg = M$ and $\tseg = T_pM$, and all the theorems in~\Cref{sec:first_order_bounds} hold for every $v \in T_pM$.

     If one then considers a Hadamard manifold with sectional curvature bounded below by $\delta < 0$, then one may apply the upper bounds in~\Cref{thm:law_of_cosines} on the whole manifold. This idea was heavily exploited in~\parencite{silvere2013stochastic} to get convergence rates for stochastic gradient descent, and in~\parencite{zhang2016riemannian} to prove convergence both stochastic and non-stochastic setting. In~\parencite{ferreria2019iteration} they used again the lower bounds whenever $\delta < 0$ to prove convergence rates of certain step-size schedulers. In~\parencite{bento2017iteration}, they use the law of cosines on manifolds of either non-negative or non-positive curvature to prove convergence rates for subgradient methods and proximal-point methods. In~\parencite{agarwal2020arc}, the lower bounds are used to prove convergence rates of certain second order method.

    In the previous presentation, we have showed that the Hadamard restriction, that is, being simply connected and with sectional curvature bounded above, is not really necessary. In particular, if we have a positive bound on the curvature, we can still apply these kind of results, at the expense of working on certain neighbourhood of the initial point. The same happens if we remove the condition of the manifold being simply connected.

    \section{Second Order Bounds for the Exponential Map}\label{sec:second_order_bounds}
    In this section, we give bounds on the Hessian of the differential of the exponential map. These kind of bounds were first given in the paper~\parencite{kaul1976schranken}. The proof can just be found in German, so we will give a self-contained presentation. Our proof uses a comparison lemma developed by Kaul, and then streamlines all the other technical tools, considerably simplifying the proof and obtaining tighter explicit bounds.

    \subsection{The differential equation}
    We start by giving a technical remark on the nature of the Hessian of the exponential map. The differential of the exponential for a point $(p,v) \in TM$ is a linear map of the form
    \[
        \deffun{\pa{\dif\exp_p}_v : T_pM -> T_{\exp_p(v)}M;}.
    \]
    For this reason, it can be seen as a section of the bundle $T^\ast_pM \tensor \exp_p^\ast(TM)$ over $T_pM$. We have connections $\connflat$ and $\conn$ on $T^\ast_pM$ and $\exp_p^\ast(TM)$ respectively. Any two connections induce a connection on a tensor product bundle so that the Leibnitz rule holds. In the concrete example of the exponential map, this Leibnitz rule for vector fields $\overline{W}_1, \overline{W}_2$ on $\tseg$, takes the form
    \[
        \conn_{\dif \exp_p(\overline{W}_1)} \pa{\dif \exp_p(\overline{W}_2)}  = \pa{\conn_{\overline{W}_1} \dif \exp_p}(\overline{W}_2) + \dif\exp_p\pa{\connflat_{\overline{W}_1}\overline{W}_2}.
    \]
    Since $\exp_p$ is a diffeomorphism between $\tseg$ and $\seg$, writing $W_i = \dif \exp_p(\overline{W}_i)$ for the pushforward of a vector from $\tseg$ to $\seg$, we see that the Hessian of the exponential map can be put in terms of the Christoffel symbols in normal coordinates as
    \begin{equation}\label{eq:conn_minus_conn_flat}
        \pa{\conn_{\overline{W}_1} \dif \exp_p}(\overline{W}_2) = \conn_{W_1} W_2 - \connflat_{\overline{W}_1}\overline{W}_2 = \Gamma(W_1, W_2) = \Gamma_{i,j}^kW_1^iW_2^j\partial_k.
    \end{equation}
    In other words, the Hessian of the exponential map is exactly the Christoffel symbols of the connection in normal coordinates at $p$ evaluated at the pushforward of the vector along $\exp_p$. In particular, this shows that the Hessian of the exponential map is a symmetric bilinear map of the form
    \[
        \deffun{\pa{\conn\dif\exp_p}_v : T_pM \times T_pM -> T_{\exp_p(v)}M;}.
    \]

    As we did for the first order bounds, our strategy to bound this quantity will be to deduce a differential equation for $\conn \dif \exp_p$ and then bound the norm of the solutions of this equation. In this case, we will have to differentiate a second time through a second variation of a geodesic. Given that this second derivative will not be in the direction of the geodesic, it will not be enough to have vector fields along the geodesic---we will have to have them defined also in the direction in which we want to differentiate them.

    \begin{proposition}\label{prop:diff_equation_second_order}
        Let $(M, \gm)$ be a Riemannian manifold and let $\gamma$ be a geodesic with initial unit vector $v \in T_pM$. For two vectors $w_1, w_2 \in T_pM$ perpendicular to $v$, the vector field along $\gamma$
        \[
            K(t) \defi \pa{\conn \dif \exp_p}_{tv}(tw_1, tw_2)
        \]
        satisfies the following second order inhomogeneous linear differential equation along $\gamma$
        \[
            \ddot{K} + R(K, \dgamma)\dgamma + Y = 0 \mathrlap{\qquad K(0) = 0, \dot{K}(0) = 0}
        \]
        where $Y$ is the vector field along $\gamma$ given by
        \[
            Y \defi
              2 R\pa{J_1, \dgamma}\dot{J}_2
            + 2 R\pa{J_2, \dgamma}\dot{J}_1
            + \pa{\conn_{\dgamma} R}\pa{J_2, \dgamma}J_1
            + \pa{\conn_{J_2} R}\pa{J_1, \dgamma}\dgamma
        \]
        and $J_1, J_2$ are the Jacobi fields along $\gamma$ with initial conditions $J_i(0) = 0$,  $\dot{J}_i(0) = w_i$.
    \end{proposition}
    \begin{proof}
        Define the geodesic variation
        \[
            \deffun{c : [0, r] \times \pa{-\epsilon, \epsilon} \times \pa{-\epsilon, \epsilon} -> M;
                    (t,s_1, s_2) -> \exp_p\pa{t\pa{v + s_1w_1 + s_2w_2}}}
        \]
        Consider the first variation along $\gamma(t) = \exp_p(tv)$ in the direction of $s_1$ given by the family of Jacobi fields
        \[
            \tilde{J}_1(t, s)= \left.\frac{\partial c}{\partial s_1}\right\vert_{s_1=0}(t, s) = \pa{\dif \exp_p}_{t(v + sw_2)}(tw_1).
        \]
        We denote by $J_i(t)$ the Jacobi field along $\gamma$ with initial conditions $(0, w_i)$. In particular, we have that $J_1(t) = \tilde{J}_1(t,0)$, so $\tilde{J}_1$ is an extension of $J_1$ in the direction of $J_2$ so that $\conn_{J_2}\tilde{J}_1$ is well defined. Moreover, we have that $\connflat_{J_2}\tilde{J}_1 = 0$ as $\tilde{J}_1$ is constant in the direction of $J_2$, that is, in the direction of $stw_2$ for $s \in (-\epsilon, \epsilon)$ and a fixed $t$.
        For this reason, the vector field $K$ is the second variation along $\gamma$ in the directions $w_1, w_2$
        \[
            K(t) =
                \pa{\conn \dif \exp_p}_{tv}\pa{tw_1, tw_2} =
                \conn_{J_2} \tilde{J}_1 =
                \left.\frac{\partial \tilde{J}_1}{\partial s}\right\vert_{s=0}(t).
        \]

        Furthermore, $\tilde{J}_1$ satisfies the Jacobi equation in all its domain with initial conditions $(0, \frac{1}{\norm{v+sw_1}})$ so that
        \[
            \conn_{\grad r}\conn_{\grad r}\tilde{J}_1 + R(\tilde{J}_1, \grad r) \grad r = 0.
        \]
        We may then derive a differential equation for $K$ along $\gamma$ by differentiating this equation in the direction of $J_2$
        \begin{equation}\label{eq:first_deriv}
            \conn_{J_2} \conn_{\grad r} \conn_{\grad r} \tilde{J}_1 + \conn_{J_2}\pa{R(\tilde{J}_1, \grad r)\grad r} = 0.
        \end{equation}
        Note that, as we just are interested in deriving a differential equation for $K$ along $\gamma$, we will use that $\grad r \vert_\gamma = \dgamma$ and $\tilde{J}_1 = J_1$ whenever we have that a quantity just depends on the values of the vector fields involved along $\gamma$, rather than in a neighborhood  around $\gamma$ in the direction of $J_2$.

        We may put the first term of~\eqref{eq:first_deriv} in terms of $K$ by repeatedly using the definition of the curvature tensor and the fact that since $[J_2, \grad r]\vert_\gamma = 0$, we have that $\conn_{J_2}\grad r = \conn_{\dgamma}J_2 = \dot{J}_2$
        \begin{align*}
            \conn_{J_2} \conn_{\grad r} \conn_{\grad r} \tilde{J}_1
            &= \conn_{\dgamma} \conn_{J_2} \conn_{\grad r} \tilde{J}_1 + R(J_2, \dgamma)\conn_{\dgamma} J_1 \\
            &= \conn_{\dgamma} \conn_{\dgamma} K + \conn_{\dgamma}\pa{R\pa{J_2, \dgamma}J_1} + R(J_2, \dgamma)\dot{J}_1 \\
            &= \conn_{\dgamma} \conn_{\dgamma} K + \pa{\conn_{\dgamma} R}\pa{J_2, \dgamma}J_1
                + R\pa{\dot{J}_2, \dgamma}J_1
                + 2R\pa{J_2, \dgamma}\dot{J}_1.
        \end{align*}
        We can expand the second term as
        \[
            \conn_{J_2}\pa{R(\tilde{J}_1, \grad r)\grad r}
            = \pa{\conn_{J_2} R}\pa{J_1, \dgamma}\dgamma
            + R\pa{K, \dgamma}\dgamma
            + R\pa{J_1, \dot{J}_2}\dgamma
            + R\pa{J_1, \dgamma} \dot{J}_2.
        \]
        Putting everything together and using the symmetries of the curvature tensor and the first Bianchi identity we get the differential equation.

        For the initial conditions, we have that
        \[
            K(0)=\pa{\conn \dif \exp_p}_0(0,0) = 0.
        \]
        Differentiating the Hessian in the direction of $\dgamma = \frac{\dif}{\dif t}$ we have that
        \[
        \dot{K}(0) = \pa{\conn\conn \dif \exp_p}_0(v, 0, 0) + \pa{\conn \dif\exp_p}_0(w_1, 0) + \pa{\conn \dif\exp_p}_0(0, w_2) = 0.\qedhere
        \]
    \end{proof}

    \subsection{Manifolds of bounded geometry}
    By~\Cref{prop:diff_equation_second_order}, we have that the Hessian of the exponential map solves a Jacobi-like differential equation with two main differences: It is inhomogeneous and depends on the covariant derivative of the curvature tensor.

    The inhomogeneity of the differential equation will stop us from simplifying the differential equation into two first order differential equations, as we did for the Jacobi equation. In contrast, this time we will have to deal directly with the second order equation. Luckily, we have already developed most of the tools to do so.

    The fact that involves the covariant derivative of the curvature tensor is a more intrinsic difference. Since the sectional curvature completely defines the curvature tensor, bounds on the sectional curvature can be translated into bounds on the norm of the curvature tensor. One question that naturally arises is whether these $0$-th order bounds also give first order bounds. As one may expect, this is not the case.

    \begin{example}[Manifold of bounded $0$-th order and unbounded $1$-st order]
        Consider a rotationally symmetric surface of the form $\dif r^2 + \rho(r)^2\dif \theta^2$ on the cylinder $(0,1) \times \SS^1$ for a function $\rho > 0$. These metrics have a particularly simple formula for the sectional curvature
        \[
            \sec(\partial_r, \partial_\theta) = \frac{R(\partial_r, \partial_\theta, \partial_\theta, \partial_r)}{\rho^2}= -\frac{\ddot{\rho}}{\rho}.
        \]
        All we have to do now is to choose a positive function $\rho$ with bounded second derivative and arbitrarily large third derivative, for example $\rho(r) = 2 + r^5 \sin(1/r)$. This metric has $\abs{\sec} < 12$, but
        \[
        \pa{\conn_{\partial_r}R}\pa{\partial_r, \partial_\theta, \partial_\theta, \partial_r}
        = D_{\partial_r}\pa{R\pa{\partial_r, \partial_\theta, \partial_\theta, \partial_r}}
        -2R\pa{\partial_r, \conn_{\partial_r}\partial_\theta, \partial_\theta, \partial_r}
        = \dot{\rho}\ddot{\rho} - \rho \dddot{\rho}
        \]
        which is unbounded as $r \to 0$.
    \end{example}

    This example motivates the following definition.
    \begin{definition}[Manifold of bounded geometry]\label{def:bounded_geometry}
        Let $(M, \gm)$ be a Riemannian manifold. We say that a manifold has $k$-bounded geometry if for every $i = 0, \dots, k$ there exist constants $C_i \geq 0$ such that
        \[
            \norm{\conn^i R} \leq C_i
        \]
        and the injectivity radius is uniformly bounded below by a positive constant
        \[
            \inj \defi \inf_{p \in M}\inj(p) > 0.
        \]
    \end{definition}

    \begin{remark}
        A manifold has $0$-bounded geometry if it has bounded sectional curvature above and below and positive injectivity radius. As one may expect, given that the $1$-bounded geometry is given by the Christoffel symbols, which are given by directional derivatives of the metric $\gm$ in normal coordinates, the condition on the boundedness of the derivatives of the curvature tensor is equivalent to asking for the entries of the metric tensor to be bounded in normal coordinates. This was first proved in~\parencite{eichhorn1991boundedness}.
    \end{remark}

    It is clear that any compact manifold with any given metric will be of $k$-bounded geometry for every $k \in \mathbb{N}$. A natural question would be whether this condition puts any constraint in the topology of the manifold. This was answered in the negative in~\parencite{greene1978complete}.
    \begin{theorem}[Greene]
        Every differentiable manifold admits a complete metric of bounded geometry.
    \end{theorem}

    Given that we usually have access to upper and lower bounds on the sectional curvature, and given the form that takes the covariant derivatives in~\Cref{prop:diff_equation_second_order}, we give the following definition, which is just a generalisation of that given in the introduction.
    \begin{definition}\label{def:weird_bounded_geometry}
        We say that a Riemannian manifold $(M, \gm)$ has $(\delta, \Delta, \Lambda)$-bounded geometry if $\inj > 0$ and
        \begin{gather*}
            \delta \leq \sec \leq \Delta\\
            \norm{\pa{\conn_u R}(w, u)w + \pa{\conn_w R}\pa{w, u}u} \leq 2\Lambda\norm{u}^2\norm{w}^2 \mathrlap{\qquad \forall p\in M\quad u,w \in T_pM.}
        \end{gather*}
    \end{definition}

    \begin{remark}
        If a manifold is of $(\delta, \Delta, \Lambda)$-bounded geometry according to the definition in the introduction, it is of $(\delta, \Delta, \Lambda)$-bounded geometry according to this definition, but not necessarily the other way around. In the proofs, we will just use this definition, as it is all we need, and it also heavily simplifies the computations when one wants to estimate the actual constant $\Lambda$ for a given manifold.
    \end{remark}

    \begin{remark}
        We will not use the fact that these manifolds have uniformly bounded injectivity radius in this paper, but this fact turns out to be crucial when proving other flavour of theorems in optimization of manifolds that involve $\exp^{-1}_p$ or parallel transport along geodesics or a retraction.
    \end{remark}

    \subsection{Curvature bounds}
    We start by recalling the following lemma that gives a closed formula for the curvature tensor of manifolds of constant sectional curvature. In its more general form, it says that the Riemannian curvature tensor is completely specified by the sectional curvature.

    \begin{lemma}\iftoggle{arxiv}{}{[Riemann, 1854]}\label{lemma:constant_curvature}
        Let $(M, \gm)$ be a Riemannian manifold of constant sectional curvature $\kappa \in \RR^n$. The curvature tensor takes the form
        \[
            R(x,y)z \defi R_\kappa(x, y)z = \kappa\pa{\scalar{z, y}x - \scalar{z, x}y} \mathrlap{\qquad \forall x,y,z \in T_pM.}
        \]
    \end{lemma}
    \iftoggle{arxiv}{}{%
    \begin{proof}
        Define
        \begin{align*}
            R_\kappa(x, y)z &= \kappa\pa{\scalar{z, y}x - \scalar{z, x}y} \\
            R_\kappa(x, y, z, w) &= \scalar{R_\kappa\pa{x,y}z, w}
        \end{align*}
        These tensors are just the $(1,3)$ and $(0,4)$ forms of the same tensor. It is direct to see that this tensor have the same symmetries as a curvature tensor, namely, it is skew symmetric in the first two components and in the last two components, satisfies the Bianchi identity, and it is symmetric between the first two and last two components. A tensor that has these symmetries is called a \emph{curvature-like tensor}. The tensor $D = R - R_\kappa$ is another curvature-like tensor, and we find that, since $(M, \gm)$ has constant sectional curvature $\kappa$, $D(v, w, w, v) = 0$.

        We now just have to prove that for a curvature-like tensor, the sectional curvature completely determines the full tensor, but this is direct applying polarization as
        \[
            \left.\frac{\partial^2}{\partial s \partial t}\right\vert_{(0,0)}
            \cor[\Big]{
            D(x+tw, y+sz, y+sz, x+tw) -
            D(x+tz, y+sw, y+sw, x+tz)
            } = 6D(x,y,z,w).\qedhere
        \]
    \end{proof}
    }

    To be able to bound the terms in the differential equation, we will need some estimates on the norm of the curvature tensor in terms of the sectional curvature. These are a generalization of Berger's lemma~\parencite{berger1960sur}. The first three inequalities are announced---the last one with a worse constant---in~\parencite{kaul1976schranken} citing for their proof a monograph by Karcher that was never published. We provide original proofs for these inequalities. The third inequality appears in~\parencite{karcher1970short}, but with a different proof. The last inequality is entirely from~\parencite{karcher1970short}.

    \begin{proposition}\label{prop:curvature_bounds}
        Let $(M, \gm)$ be a Riemannian manifold with bounded sectional curvature $\delta \leq \sec \leq \Delta$ and define
        \[
            \epsilon = \frac{\Delta - \delta}{2} \qquad \mu = \frac{\Delta + \delta}{2} \qquad K = \max\set{\abs{\Delta}, \abs{\delta}}.
        \]
        Given a point $p \in M$, we have the following inequalities
        \begin{align}
            \abs{R(x, y, y, w) - R_\mu(x, y, y, w)} &\leq \epsilon\norm{x}\norm{y}^2\norm{w}\mathrlap{\qquad \forall x, y, w \in T_pM} \label{eq:inequality_R_3}\\
            \norm{R(x, y, z, w) - R_\mu(x, y, z, w)} &\leq \frac{4}{3}\epsilon\norm{x}\norm{y}\norm{z}\norm{w}\mathrlap{\quad \forall x, y, z, w \in T_pM}\label{eq:inequality_R_4}
        \end{align}
        The constants $1$ and $\frac{4}{3}$ are tight.

        Furthermore, if $e \in T_pM$ is a unitary vector and we define $\normal{R}(x,y)z$ as the component of the curvature tensor perpendicular to $e$, we have that for every $y,z$ perpendicular to $e$
        \begin{align}
            \norm{\normal{R}(e,y)z} &\leq \frac{4}{3}\epsilon\norm{y}\norm{z} \label{eq:bound_R_normal}\\
            \norm{R(e,y)z} &\leq 2K.\label{eq:bound_R_all}
        \end{align}
    \end{proposition}
    \begin{proof}
        We may assume that the vectors $x,y,z,w$ are of norm $1$ by linearity. For~\eqref{eq:inequality_R_3}, since the curvature tensor is skew-symmetric in the first two components and the last two components, we may assume that $x,w$ are perpendicular to $y$. Fix a vector $y$ and consider the bilinear form
        \[
            T_y(x,w) = R(x,y,y,w) - R_\mu(x,y,y,w).
        \]
        This application is symmetric, and as such, it attains its maximum at an eigenvector $x_1$ of norm one orthogonal to $y$ so that for every $x$
        \[
            \abs{T_y(x,w)} \leq \abs{R(x_1,y,y,x_1) - R_\mu(x_1,y,y,x_1)} = \abs{\sec(x_1, y) - \mu} \leq \epsilon.
        \]

        For the second inequality, by linearity we can assume that all the vectors have the same norm. By polarization we have that
        \begin{align*}
            6R(x,y,z,w) &= R(x, y+z, y+z, w) - R(x, y-z, y-z, w) \\
                        &-R(y, x+z, x+z, w) + R(y, x-z, x-z, w).
        \end{align*}
    We then apply the triangle inequality to the expression for $(R-R_\mu)(x,y,z,w)$ together with~\eqref{eq:inequality_R_3} to get
        \[
            6\abs{(R-R_\mu)(x,y,z,w)} \leq \epsilon\pa[\Big]{\norm{x}\norm{w}\pa{\norm{y+z}^2 + \norm{y-z}^2} + \norm{y}\norm{w}\pa{\norm{x+z}^2+\norm{x-z}^2}}
        \]
        and using the parallelogram law together with the fact that all the vectors have the same norm, we get that
        \[
            \norm{x+y}^2 + \norm{x-y}^2 = 2(\norm{x}^2 + \norm{y}^2) = 4\norm{x}\norm{y}
        \]
        and the second inequality follows.

        These two inequalities are tight on $\CC P^n$ seen as a real manifold~\parencite{karcher1970short}.

        The bound on the normal part of the curvature tensor follows directly from~\eqref{eq:inequality_R_4}, since $R_\mu(e,y,z,w) = 0$ whenever $y,z,w$ are perpendicular to $e$, so choosing $w$ as the unitary vector in the direction of $\normal{R}(e,y)z$ we get that
        \[
            \norm{\normal{R}(e,y)z} = \abs{R(e,y,z,w)} \leq \frac{4}{3}\epsilon\norm{y}\norm{z}.
        \]

        The last inequality is proved in~\parencite{karcher1970short}.
    \end{proof}

    \subsection{A comparison lemma}
    We shall proceed in a very similar way to how we did in the case of the first order equation. We will simplify the differential equation to one in one dimension, and there, we will use a comparison theorem for functions of real variable. The only difference is that, in this case, we will not be able to simplify the computations to a Riccati equation. We start by proving the second order version of the Riccati comparison lemma but for the Jacobi equation.
    \begin{lemma}[Jacobi comparison lemma]\label{lemma:comparison_edo}
        Fix a real number $\kappa \in \RR$, and let $\deffun{f, g : [0, r] -> \RR^+;}$ be two functions such that
        \[
            \ddot{f} + \kappa f \leq \ddot{g} + \kappa g \mathrlap{\qquad f(0) = g(0), \dot{f}(0) \leq \dot{g}(0).}
        \]
        Then $f \leq g$ on $[0, \min\set{r, \pi_{\kappa}}]$.
    \end{lemma}
    \begin{proof}
        Let $h = g - f$ and $\zeta = \ddot{h} + \kappa h \geq 0$. Let us show that the solution to the linear inhomogeneous initial value problem
        \[
            \ddot{h} + \kappa  h = \zeta\mathrlap{\qquad h(0) = 0,\ \dot{h}(0) \geq 0}
        \]
        is indeed positive on the given interval. Solving the equation, we find that the solution is given by
        \begin{align*}
            h(x)
            &= \sn_\kappa(x)\intf[0][x]{\sn'_\kappa(t)\zeta(t)}{t} - \sn_\kappa'(x)\intf[0][x]{\sn_\kappa(t)\zeta(t)}{t} + \dot{h}(0)\sn_\kappa(x)\\
            &= \intf[0][x]{\sn_\kappa(x-t)\zeta(t)}{t} + \dot{h}(0)\sn_\kappa(x).
        \end{align*}
        where we have used the trigonometric identity
        \[
            \sn_\kappa(x-t) = \sn_\kappa(x)\sn'_\kappa(t) - \sn'_\kappa(x)\sn_\kappa(t).
        \]
        From this we see that $h(x) \geq 0$ on $[0, \min\set{r, \pi_\kappa}]$ as $\sn_\kappa$, $\zeta$ are positive on this interval.
    \end{proof}

    We can now show how to estimate the inhomogeneous differential equation for the Hessian of the exponential, taking advantage of the fact that the initial conditions are both zero.
    \begin{proposition}[Kaul, 1976]\label{prop:comparison_second_order}
        Let $(M, \gm)$ be a Riemannian manifold with bounded sectional curvature $\delta \leq \sec \leq \Delta$.
        Let $\deffun{\gamma : [0, r] -> M;}$ be a geodesic, and let $X, Y$ be vector fields along $\gamma$ with $X, Y \perp \dgamma$ such that
        \[
            \ddot{X} + R(X, \dgamma)\dgamma = Y \mathrlap{\qquad X(0) = 0, \dot{X}(0) = 0.}
        \]
        Assume that there exists a continuous $\eta$ function such that $\norm{Y} \leq \eta$ on $[0,r]$. Then, we have that $\norm{X} \leq \rho$ on $[0, \min\set{r, \pi_{\frac{\Delta + \delta}{2}}}]$, where $\rho$ is the solution of
        \[
            \ddot{\rho}+\delta\rho = \eta \mathrlap{\qquad \rho(0) = 0, \dot{\rho}(0) = 0.}
        \]
    \end{proposition}
    \begin{proof}
        We define again the quantities
        \[
            \epsilon = \frac{\Delta - \delta}{2} \qquad \mu = \frac{\Delta + \delta}{2}.
        \]
        Fix a $t_0 \in [0, r]$ and let $E$ be the parallel vector field along $\gamma$ such that $E(t_0) = \frac{X(t_0)}{\norm{X(t_0)}}$. The function $f = \scalar{X, E}$ satisfies the differential equation
        \begin{align*}
            \ddot{f} + \mu f
            &= \scalar{\ddot{X} + \mu X, E}\\
            &= \scalar{Y - R\pa{X, \dgamma}\dgamma + R_{\mu}\pa{X, \dgamma}\dgamma, E}\\
            &\leq \norm{Y} + \epsilon \norm{X}
        \end{align*}
        where we have used~\Cref{lemma:constant_curvature} in the first equality and~\eqref{eq:inequality_R_3} for the bound.

        Define $g$ as the solution to the differential equation
        \[
            \ddot{g} + \mu g = \norm{Y} + \epsilon\norm{X} \mathrlap{\qquad g(0) = \dot{g}(0) = 0.}
        \]
        Using that $f(t_0) = \norm{X(t_0)}$ and that $f(0) = \dot{f}(0) = 0$, by~\Cref{lemma:comparison_edo}, we get that  $\norm{X(t_0)} \leq g(t_0)$, and since the definition of $g$ does not depend on the chosen $t_0$, we may do this for any $t_0 \in [0,r]$ getting that $\norm{X} \leq g$ on $[0, \min\set{\pi_\mu, r}]$. Now, Using that $\norm{X} \leq g$, we get that $g$ satisfies the inequality
        \[
            \ddot{g} + \delta g \leq \norm{Y}\leq \eta.
        \]
        So, for the solution of
        \[
            \ddot{\rho}+\delta\rho = \eta \mathrlap{\qquad \rho(0) = 0, \dot{\rho}(0) = 0.}
        \]
        we have that $g \leq \eta$ on $[0, \min\set{\pi_\delta, r}] \supseteq [0, \min\set{\pi_\mu, r}]$, getting the result.
    \end{proof}

    \subsection{A second order version of Rauch's theorem}
    We are now ready to give bounds on the Hessian of the exponential map. We first note that, since our goal is to bound the norm of the Hessian of $f \circ \exp_p$, as this Hessian is symmetric, it attains its maximum at an eigenvector. As such, we just need to bound the quantity
    \[
        \pa{\conn\dif \exp_p}_v(w,w) \mathrlap{\qquad v \in \tseg, w \in T_pM.}
    \]
    For this reason, we will start by giving bounds on the diagonal of the Hessian of the exponential map. We will then see that, since the Hessian is a symmetric bilinear map, we can leverage these bounds to give bounds on the full Hessian for any pair of vectors $w_1, w_2 \in T_pM$.

    We give the second order bounds on the exponential map for a manifold of $(\delta, \Delta, \Lambda)$-bounded geometry (\cf~\Cref{def:weird_bounded_geometry}).

    \begin{theorem}[Second order bounds for the exponential map]\label{thm:second_order_bounds}
        Let $(M, \gm)$ be a Riemannian manifold with $(\delta, \Delta, \Lambda)$-bounded geometry. For a geodesic $\deffun{\gamma : [0,r] -> M;}$ with initial unit vector $v$, and a vector $w \in T_pM$, we have that
        \begin{itemize}
        \item If $w$ is radial to $\dgamma(0)$,
            \[
                \pa{\conn\dif\exp_p}_{rv}(w, w) = 0.
            \]
        \item If $w$ is normal to $\dgamma(0)$, the radial part of the Hessian is bounded as
            \[
                \pa[\Big]{\frac{1}{r} - \frac{\sn_{4\delta}(r)}{r^2}}\norm{w}^2
                \leq
                \scalar{\pa{\conn\dif \exp_p}_{rv}\pa{w,w}, \dgamma(r)}
                \leq
                \pa[\Big]{\frac{1}{r} - \frac{\sn_{4\Delta}(r)}{r^2}}\norm{w}^2
            \]
            for $r < \pi_\Delta$ for the upper bound and $r < r_{\conjloc}(v)$ for the lower bound.

            The normal part of the Hessian is bounded for $r < \pi_{\frac{\Delta + \delta}{2}}$ as
            \[
                \norm{\normal{\pa{\conn\dif \exp_p}_{rv}\pa{w,w}}}
                \leq
                \rho(r) \norm{w}^2
            \]
            where
            \[
                \rho(t) =
                \lfrac{8}{9r^2}\sn_\delta\pa[\big]{\lfrac{t}{2}}^2
                \pa[\big]{3\Lambda\sn_\delta\pa[\big]{\lfrac{t}{2}}^2
                +2\pa{\Delta - \delta}\sn_\delta\pa{t}}.
            \]
        \end{itemize}
        Both bounds and their radii are tight in spaces of constant curvature.

        Furthermore, the radius of convergence for the normal part is tight for $\SO{n}$.
    \end{theorem}
    \begin{proof}
        As in~\Cref{prop:diff_equation_second_order}, we write
        \begin{align*}
            \tilde{J}(t, s)&= \pa{\dif \exp_p}_{t(v + sw)}\pa[\Big]{\frac{t}{r}w}\\
            J(t) &= \tilde{J}(t,0) \\
            K(t) &=\pa{\conn \dif \exp_p}_{tv}\pa[\Big]{\frac{t}{r}w, \frac{t}{r}w}.
        \end{align*}
        By linearity of the Hessian, it is enough to prove the result for $\norm{w} = 1$.

        If $w$ is radial, in~\Cref{prop:diff_equation_second_order} we have that $Y = 0$, so $K$ is a solution of the equation
        \[
            \ddot{K} + R(K, \dgamma)\dgamma = 0 \mathrlap{\qquad K(0) = 0, \dot{K}(0) = 0}
        \]
        and since it is a homogeneous second order linear equation with zero as the initial condition, its solution is $K(t) = 0$ for $t \in [0, r]$.

        If $w$ is normal, we need to bound the quantity $\conn_J \tilde{J}$. Note that this is a vector field along $\gamma$, but to take the derivative in the direction of $J$ we need to have $\tilde{J}$ defined in that direction as well. We start by bounding its radial part
        \begin{equation}\label{eq:radial_part_hessian}
            \scalar{\conn_J \tilde{J}, \grad r}
            = D_J\scalar{\tilde{J}, \grad r}- \scalar{\tilde{J}, \conn_J \grad r}
            = D_J\scalar{\tilde{J}, \grad r}- \Hess r(J, J).
        \end{equation}
        We can compute the first term directly. By Gauss's lemma we can simplify this derivative to one in $T_pM$
        \[
            D_J\scalar{\tilde{J}, \grad r}(\gamma(t))
            = \frac{1}{r}\left.\frac{\dif}{\dif s}\right\vert_{s=0}\scalar[\Big]{\frac{t}{r}w, \frac{t\pa{v+sw}}{\norm{t\pa{v +sw}}_{T_pM}}}_{T_pM}
            = \frac{t}{r^2}.
        \]
        Using the bounds on the Hessian of the distance function given in~\Cref{thm:rauch} we can bound the second term in~\eqref{eq:radial_part_hessian}. Evaluating these two quantities at $r$ and using the bounds on the size of the Jacobi fields together with the trigonometric identity
        \begin{equation}\label{eq:trig}
            \sn_\kappa(t)\sn'_\kappa(t) = \frac{\sn_\kappa(2t)}{2} = \sn_{4\kappa}(t)
        \end{equation}
        we get the bounds on the tangential part of the Hessian of the exponential.

        Finally, for its normal part, consider the differential equation given by~\Cref{prop:diff_equation_second_order}. Since $\scalar{K, \dgamma}' = \scalar{\dot{K},\dgamma}$, the radial (resp.\ normal) part of the derivative is the derivative of the radial (resp.\ normal) part. For this reason, we have that
        \[
            \pa{\normal{K}}'' + R(\normal{K}, \dgamma)\dgamma = -\normal{Y}.
        \]
        Therefore, we just have to bound $\norm{\normal{Y}}$ to be able to use~\Cref{prop:comparison_second_order} and finish.

        We start by giving a bound on the norm of $\dot{J}$. Using that $\dot{J} = \conn_J\grad r = \Hess r(J)$,
        \begin{equation}\label{eq:bound_derivative_jac}
            \norm{\dot{J}} \leq \norm{\Hess r} \norm{J}.
        \end{equation}
        We can then bound the norm of $\normal{Y}$ as
        \begin{align*}
            \norm{\normal{Y}}
            &\leq
             \norm{\pa{\conn_{\dgamma} R}\pa{J, \dgamma}J +
             \pa{\conn_J R}\pa{J, \dgamma}\dgamma}
              + 4 \norm{\normal{R}\pa{J, \dgamma}\dot{J}}\\
            &\leq 2\Lambda \norm{J}^2 + \frac{8(\Delta - \delta)}{3}\norm{J}\norm{\dot{J}}\\
            &\leq \pa[\Big]{2\Lambda + \frac{8(\Delta - \delta)}{3}\ct_\delta(t)}\norm{J}^2\\
            &\leq \pa[\Big]{2\Lambda + \frac{8(\Delta - \delta)}{3}\ct_\delta(t)}\frac{\sn_\delta(t)^2}{r^2}\\
            &= \frac{2}{r^2}\pa[\Big]{\Lambda\sn_\delta(t)^2 + \frac{2(\Delta-\delta)}{3}\sn_\delta(2t)}.
        \end{align*}
        where we have used~\eqref{eq:bound_R_normal}---as since $J$ is perpendicular to $\dgamma$, so is $\dot{J}$---to bound the normal part of the curvature tensor. We have also used the first order bound on the differential of the exponential for perpendicular initial conditions to bound the norm of the Jacobi fields and~\eqref{eq:bound_derivative_jac} to bound their derivative. In the last equality we have used~\eqref{eq:trig} again.

        The result follows by noting that $\rho$ is the solution to the differential equation
        \[
            \ddot{\rho} + \delta \rho = \eta \mathrlap{\qquad \rho(0) = \dot{\rho}(0)=0}
        \]
        where $\eta$ is the bound on $\norm{\normal{Y}}$ and applying~\Cref{prop:comparison_second_order}.

        We will see that the radius for the bound of the normal part is tight in the case of $\SO{n}$ in~\Cref{ex:son}.
    \end{proof}

    The first thing to note is that these bounds go to zero as $r$ tends to zero. This is exactly what we expect, as the Christoffel symbols in normal coordinates vanish at the origin.

    Bounds on the whole Hessian can be easily deduced from those presented here via polarization.

    The bounds in this theorem provide a notable improvement compared to the best bounds previously known (\cf~\parencite{kaul1976schranken}). For one, these bounds are tighter. We also have that these bounds are explicit, compared to the previous bounds, which were given in terms of a solution of a differential equation that did not have an explicit integral. We also notably simplified the technical tools necessary to get to these bounds.

    We finish this section by giving a bound on the full Hessian of the exponential.
    \begin{theorem}[Bounds on the Full Hessian]\label{thm:full_second_order_bounds}
        Let $(M, \gm)$ be a Riemannian manifold with $(\delta, \Delta, \Lambda)$-bounded geometry. For a geodesic $\deffun{\gamma : [0,r] -> M;}$ with initial unit vector $v$, $r < \pi_{\frac{\Delta + \delta}{2}}$, and any two vectors $w_1, w_2 \in T_pM$, we have that
        \[
            \norm{\pa{\conn\dif \exp_p}_{rv}\pa{w_1,w_2}} \leq
                \lfrac{8}{3r^2}\sn_\delta\pa[\big]{\lfrac{r}{2}}^2
                \pa{\Lambda\sn_\delta\pa[\big]{\lfrac{r}{2}}^2
                +2\max\set{\abs{\Delta}, \abs{\delta}}\sn_\delta\pa{r}}\norm{w_1}\norm{w_2}.
        \]
        Furthermore, the radius $\pi_{\frac{\Delta + \delta}{2}}$ is tight for $\SO{n}$.
    \end{theorem}
    \begin{proof}
        To bound the full Hessian we first see that it is just enough to bound its diagonal part. For any symmetric bilinear form and any two vectors we have the polarisation formula
        \[
            4\Phi(u,v) = \Phi(u+v, u+v) - \Phi(u-v, u-v).
        \]
        By linearity, we may assume that $\norm{u} = \norm{v} = 1$. Taking absolute values and applying the triangle inequality and Cauchy-Schwarz, we get that
        \[
            4\norm{\Phi(u,v)} \leq \norm{\Phi \vert_{\operatorname{diag}}}\pa{\norm{u+v}^2 + \norm{u-v}^2} = 4\norm{\Phi \vert_{\operatorname{diag}}}
        \]
        where $\norm{\Phi \vert_{\operatorname{diag}}}$ is the operator norm of the application $u \mapsto \Phi(u,u)$. For this reason, it is enough to bound the map $u \mapsto \pa{\conn\dif\exp_p}_{tv}\pa{u,u}$.

        Let $w$ be a vector normal to $v$. As we did in~\Cref{thm:second_order_bounds}, we consider the differential equation for
        \[
            K(t) =\pa{\conn \dif \exp_p}_{tv}\pa[\Big]{\frac{t}{r}w, \frac{t}{r}w}.
        \]
        This is exactly the same differential equation that we had for $\normal{K}$, only that rather than having to bound the normal part of the curvature tensor, we bound the full curvature tensor. For that we use~\eqref{eq:bound_R_all}, and following with the bounds as we did in~\Cref{thm:second_order_bounds} and solving the resulting equation we get that for $w$ perpendicular to $v$, the norm of $K$ is bounded by the solution of the differential equation
        \[
            \ddot{\rho} + \delta \rho
        = \frac{2}{r^2}\pa{\Lambda\sn_\delta(t)^2 + 2\max\set{\abs{\Delta},\abs{\delta}}\sn_\delta(2r)}
             \mathrlap{\qquad \rho(0) = \dot{\rho}(0)=0}
        \]
        which is solved by
        \[
            \rho(t) =
                \lfrac{8}{3r^2}\sn_\delta\pa[\big]{\lfrac{t}{2}}^2
                \pa{\Lambda\sn_\delta\pa[\big]{\lfrac{t}{2}}^2
                +2\max\set{\abs{\Delta}, \abs{\delta}}\sn_\delta\pa{t}}.\qedhere
        \]
    \end{proof}

    \begin{remark}[Tighter bounds]
        Another way to obtain bounds on the full Hessian of the exponential would be to take the bounds from~\Cref{thm:second_order_bounds} for the normal and parallel part of the Hessian and using the Cauchy-Schwarz inequality to get bounds of the form
        \[
            \norm{\pa{\conn\dif \exp_p}_{rv}\pa{w_1,w_2}} \leq
                \sqrt{\sigma^2 + \rho^2}\norm{w_1}\norm{w_2}.
        \]
        Where $\rho$ is the bound of the normal part defined in~\Cref{thm:second_order_bounds} and
        \[
            \sigma = \frac{1}{r^2}\max\set{\abs{r-\sn_{4\delta}(r)}, \abs{r-\sn_{4\Delta}(r)}}
        \]
        that is, $\sigma$ is the bound on the norm of the parallel part of the Hessian.

        This bound, although tighter in most specific examples, might be more difficult to manipulate and lacks the simplicity of that presented in the theorem above.
    \end{remark}

\subsection{Concrete second order bounds}\label{sec:concrete_bounds}
\subsubsection{Constant curvature: Sphere, hyperbolic space, and Euclidean space}
    The simplest family to evaluate these bounds in is that of the spaces of constant curvature such as the sphere, the hyperbolic plane, the Euclidean space or the flat torus.

    For a manifold of constant curvature $\kappa$, since the curvature is constant, the derivative of the curvature tensor is zero. If we further assume that their injectivity radius is positive\footnote{This is not really used to prove the bounds, but it is part of the definition of bounded geometry.}, as is the case of the hyperbolic space and the sphere, we have that they have $(\kappa, \kappa, 0)$-bounded geometry. Instantiating the bounds for these spaces we see that
    \begin{align*}
        \scalar{\pa{\conn\dif \exp_p}_{rv}\pa{w,w}, \dgamma(r)}
        &=
        \pa[\Big]{\frac{1}{r} - \frac{\sn_{4\kappa}(r)}{r^2}}\norm{w}^2 \\
        \norm{\normal{\pa{\conn\dif \exp_p}_{rv}\pa{w,w}}}
        &= 0
    \end{align*}
    for $r < \pi_\kappa$. As a first sanity check, we see that for the flat case $\kappa = 0$, the radial part is exactly equal to zero, as de whole Hessian of the exponential map is everywhere zero in this case---the second derivative of an affine function is zero.

    To check that this solution is actually correct, we can solve the differential equation in~\Cref{prop:diff_equation_second_order} for these spaces for
    \[
        K(t) \defi \pa{\conn \dif \exp_p}_{tv}(tw_1, tw_2).
    \]
    Using~\Cref{lemma:constant_curvature}, together with the fact that, since the sectional curvature is constant, $\conn R = 0$, and the trigonometric identity
    \[
        \sn_\kappa(t)\sn'_\kappa(t) = \frac{\sn_\kappa(2t)}{2} = \sn_{4\kappa}(t)
    \]
    we have that the differential equation for the constant curvature case is given by
    \[
        \ddot{K}(t) + \kappa \normal{K}(t) = 4\kappa\sn_{4\kappa}(t)\dgamma(t)\mathrlap{\qquad K(0) = 0,\ \dot{K}(0) = 0}
    \]
    We can split this equation into its normal and radial part. The normal part is clearly zero, since the right hand side is zero, and the remaining equation is linear with zero as the initial condition. For the radial part, setting $x = \scalar{K, \dgamma}$ we get
    \[
        \ddot{x}(t) = 4\kappa\sn_{4\kappa}(t)\mathrlap{\qquad x(0) = 0,\ \dot{x}(0) = 0.}
    \]
    Finally, $\frac{x(r)}{r^2} = \frac{r - \sn_{4\kappa}(r)}{r^2}$ is exactly the value announced before.

\subsubsection{Locally symmetric spaces: Orthogonal group, Symmetric positive definite matrices, and Grassmannian manifold}
Locally symmetric spaces define a large family of particularly well-behaved Riemannian manifolds. Examples of these spaces are the flat torus, the orthogonal group (or any compact Lie group with a bi-invariant metric), the space of symmetric positive definite matrices, the Grassmannian, the oriented Grassmannian and the hyperbolic Grassmannian.\footnote{All these manifolds are to be regarded as Riemannian manifolds with the metric inherited from their quotient structure}

    We recall the algebraic definition of a locally symmetric space.
    \begin{definition}
        A Riemannian manifold $(M, \gm)$ is \emph{locally symmetric} if the curvature tensor is covariantly constant, that is, $\conn R = 0$.
    \end{definition}

    Locally symmetric spaces were introduced and studied by Cartan in 1926, who also gave a complete classification of these in 1932.

    The most notable examples of locally symmetric spaces are symmetric spaces which are one of the most important families of real manifolds in Riemannian geometry. These were intensively studied by Sigurður Helgason~\parencite{helgason1978differential}.
    \begin{definition}
        A Riemannian manifold $(M, \gm)$ is a \emph{symmetric space} if, for every point $p \in M$ there exists an involutive isometry $\sigma_p$ that fixes $p$, that is
        \[
            \sigma_p(p) = p \qquad \pa{\dif \sigma_p}_0 = - \Id.
        \]
    \end{definition}
    As the name implies, one may prove that symmetric spaces are indeed locally symmetric spaces.

    For these manifolds we have the following result.
    \begin{proposition}
        A symmetric space with bounded sectional curvature $\delta \leq \sec \leq \Delta$ is of $(\delta, \Delta, 0)$-bounded geometry.
    \end{proposition}
    \begin{proof}
        Since it is a locally symmetric space, we have that $\conn R = 0$, so the first order bound is clear. Since symmetric spaces are Riemannian homogeneous spaces, for every two points there exists an isometry taking one to the other. As such, the injectivity radius is constant throughout the manifold and thus, positive.
    \end{proof}

    From this, we get that we just need to compute bounds on the sectional curvature for these manifolds in order to give second order bounds for the exponential map. Now, bounds on the sectional curvature of these manifolds are well known. We give here some examples that are particularly useful in the context of optimization.

    \begin{example}[The special orthogonal group]\label{ex:son}
        The sectional curvature for a Lie group with a bi-invariant metric, after identifying any pair of tangent vector with vectors in the Lie algebra, is given for a pair of orthonormal vectors $X, Y \in \glie$
    \[
        \sec(X,Y) = \frac{1}{4}\norm{\cor{X, Y}}^2.
    \]
    In the special case of $\SO{n}$ for $n > 2$, we have that $\glie \iso \Skew{n}$.
    It is clear that the sectional curvature is non-negative for any bi-invariant metric. In the case of $\SO{n}$ this bound is tight.

    Consider the bi-invariant metric given by the Frobenius norm---the metric inherited from $\RR^n$. For the upper bounds, if we work with a matrix Lie group, it is enough to bound the norm of $XY - YX$ for matrices $X, Y$ of Frobenius norm $1$. In the general case, this inequality is called the Böttcher-Wenzel inequality and it reads
    \[
        \norm{\cor{X,Y}} \leq 2\norm{X}\norm{Y} \mathrlap{\qquad \forall X, Y \in \M{n}.}
    \]
    For a review of this inequality and a particularly clean proof see~\parencite{lu2012remarks}.

    For the case of $\SO{n}$, that is, when $X, Y$ are skew-symmetric, this inequality can be improved~\parencite{bloch2005commutators}
    \[
        \norm{\cor{X,Y}} \leq \norm{X}\norm{Y} \mathrlap{\qquad \forall X, Y \in \Skew{n}.}
    \]
    For $n = 3$ the constant can be further improved to $\frac{1}{2}$. These constants are tight.

    Wrapping all this together, we get that the bounds for $\SO{n}$ with $n > 2$ are given by
    \begin{alignat*}{2}
        0
        \leq
        \scalar{\pa{\conn\dif \exp_p}_{rv}\pa{w,w}, \dgamma(r)}
        &\leq
        \pa[\Big]{\frac{1}{r} - \frac{\sin\pa{r}}{r^2}}\norm{w}^2
        &&\qquad r < 2\pi\\
        \norm{\normal{\pa{\conn\dif \exp_p}_{rv}\pa{w,w}}}
        &\leq
        \frac{r}{9}\norm{w}^2
        &&\qquad r < 2\sqrt{2}\pi.
    \end{alignat*}
    Note that the radius of definition of the normal part of the Hessian are tight, as the conjugate radius of $\SO{n}$ is exactly $2\sqrt{2}\pi$. This can be seeing by noting that the exponential of matrices is not full rank at matrices with two eigenvalues that are $2\pi i $ apart.

    In particular, we can give a bound on the full Hessian by bounding the derivative of $\sqrt{\sigma^2 + \rho^2}$ where $\sigma$ and $\rho$ are the bounds on the tangential and normal part of the Hessian. This gives
    \[
        \norm{\pa{\conn \dif \exp_p}_{rv}(w_1,w_2)} \leq \frac{2r}{9}\norm{w_1}\norm{w_2} \qquad{r < 2\pi}.
    \]
    We can see that~\Cref{thm:full_second_order_bounds} gives us a coarser bound, but on the other hand, the bound is defined on a larger radius. In particular we get
    \[
        \norm{\pa{\conn \dif \exp_p}_{rv}(w_1,w_2)} \leq \frac{r}{3}\norm{w_1}\norm{w_2} \qquad{r < 2\sqrt{2}\pi}.
    \]
    This radius is tight, as there are points that are $2\sqrt{2}\pi$ apart of any given point which are conjugate to it. This can be seen by computing the eigenvalues of the differential of the exponential map (see for example~\cite[][Theorem $D.2$]{lezcano2019cheap}).

    Better bounds are possible by using a tighter version of the bounds on the curvature tensor at the expense of having an uglier numeric constant. For example, the $\frac{1}{3}$ constant can be improved this way to $\frac{\sqrt{61}}{36}$.
    \end{example}

    These same ideas can be generalized to symmetric spaces. We will use this to compute bounds for the Hessian of the exponential map in the Grassmannian.

    \begin{example}[Sectional Curvature of a Symmetric space]
        Let $G/H$ be a symmetric space. Since a symmetric space is a Riemannian homogeneous space, we just need to bound the sectional point at one point, as the space has the same sectional curvature at every point. Denote by $\mlie = \normal{\hlie}\subset \glie$ the orthogonal complement of the Lie algebra of $H$ with respect to the metric at the identity in $G$. This complement to $\hlie$ is not a Lie algebra itself and in fact $[\mlie, \mlie] \subset \hlie$, as $G/H$ is a symmetric space. This set $\mlie$ may be identified isometrically with the tangent space at $\pi(e)$, the projection of the identity element $e \in G$, in other words, the map
        \[
        \deffun{\pa{\dif \pi}_e\vert_\mlie : \mlie -> T_{\pi(e)}G/H;}
        \]
        is a linear isometry. Via this identification we may treat vectors $X, Y \in T_{\pi(e)}G/H$ as vectors $\overline{X}, \overline{Y} \in \mlie \subset \glie$. Now, by O'Neill's formula~\parencite[][Chapter 7, Thm 47]{oneill1966fundamental}, the sectional curvature for these manifolds has a particularly simple formula for orthonormal vectors $X, Y \in T_{\pi(e)}G/H$
        \[
            \sec_{G/H}(X, Y) = \sec_G(\overline{X}, \overline{Y}) + \frac{3}{4}\norm{\cor{\overline{X}, \overline{Y}}}^2
        \]
        where we have used that $[\mlie, \mlie]\subset \hlie$.

        If $G$ is a Lie group with a bi-invariant metric we say that $G/H$ is a \emph{normal symmetric space}. For a normal metric, the sectional curvature simplifies to
        \[
            \sec_{G/H}(X, Y) = \norm{\cor{\overline{X}, \overline{Y}}}^2.
        \]
        Trough the study of symmetric spaces of non-compact type and their duality, it is not difficult to prove that the sectional curvature on any symmetric space is either the norm or minus the norm of the Lie bracket of a pair of vectors, although we will not prove that as we will not need it.
    \end{example}

    \begin{example}[The real Grassmannian]
        The real Grassmannian as a symmetric space is given by the quotient $\Gr{n,k} = \SO{n} / \pa{\SO{k} \times \SO{n-k}}$, where the metric on $\SO{n}$ is the bi-invariant metric generated by the scalar product $\scalar{X, Y} = \frac{1}{2}\tr\pa{\trans{X}Y}$ on the Lie algebra. For this symmetric space we have that
        \begin{gather*}
            \hlie =
            \so{k}\oplus\so{n-k} =
            \set[\Big]{
                \begin{pmatrix}
                    B & 0 \\
                    0 & C
                \end{pmatrix}
                | B \in \Skew{k}, C \in \Skew{n-k}}\\
            \mlie = \set[\Big]{
                        \begin{pmatrix}
                            0 & A \\
                            -\trans{A} & 0
                        \end{pmatrix}
                        | A \in \M{n-k, k}}
        \end{gather*}
        Note that the metric is chosen so that for $\overline{X} \in \mlie$ we have that $\norm{\overline{X}} = \norm{X}$, where $X \in \M{n-k,k}$. This is so that this norm agrees with the usual norm in the projective plane as $\Gr{n,1} \iso \RR P^n$.

        Since the metric on $G = \SO{n}$ is bi-invariant, the Grassmannian is a normal symmetric space, so the sectional curvature is given by the norm of the commutator of elements in $\mlie$ . Using the bound on the norm of the Lie bracket for skew-symmetric matrices, we would get
        \[
            0 \leq \sec(X, Y) \leq 4.
        \]
        As it can be seen by~\parencite[][Lemma 2.5]{ge2014ddvv}, these bounds are not tight. They can be refined as announced in~\parencite[][Theorem 3a]{wong1968sectional} and proved in~\parencite[][p.292]{hildebrandt1980harmonic} via a simple application of Cauchy-Schwarz as
        \[
            0 \leq \sec(X, Y) \leq 2.
        \]
        Using these bounds, we get analogous bounds for the Hessian of the exponential map for the Grassmannian,
        \begin{alignat*}{2}
            0
            \leq
            \scalar{\pa{\conn\dif \exp_p}_{rv}\pa{w,w}, \dgamma(r)}
            &\leq
            \pa[\Big]{\frac{1}{r} - \frac{\sin\pa{2\sqrt{2}r}}{2\sqrt{2}r^2}}\norm{w}^2
            &&\qquad r < \frac{\pi}{\sqrt{2}}\\
            \norm{\normal{\pa{\conn\dif \exp_p}_{rv}\pa{w,w}}}
            &\leq
            \frac{8r}{9}\norm{w}^2
            &&\qquad r < \pi.
        \end{alignat*}
        Note that for the real Grassmannian $\inj = \frac{\pi}{2}$~\parencite{kozlov2000geometry}, so the radii of these equations should be enough for any practical purposes. As in the case of $\SO{n}$, it is also direct to get linear bounds on the full Hessian.
    \end{example}

\section{Convergence rates of dynamic trivializations}\label{sec:weakly_convex}
    In this section, we use the bounds developed in the last two sections to prove convergence of several descent algorithms on Riemannian manifolds.

    \subsection{Weakly Convex Optimization on Manifolds}
    We now have all the necessary tools to be able to complete the program stated in the introduction. In particular, we can prove the convergence of the dynamic trivialization framework for any stopping rule. We first recap the results of the previous two sections in the following proposition.

    \begin{proposition}\label{prop:weak_convexity}
        Let $(M, \gm)$ be a connected and complete Riemannian manifold of $(\delta, \Delta, \Lambda)$-bounded geometry and let $f$ be an $\alpha$-weakly convex function on it. Fix a point $p \in M$, and let $\mathcal{X} \subset \seg$ be a subset of $M$ with $\diam(\mathcal{X}) = r \leq \pi_{\frac{\Delta + \delta}{2}}$, and at least one critical point of $f$ in it. Denote the pullback of $\mathcal{X}$ under the exponential as $\overline{\mathcal{X}} \defi \exp^{-1}_p(\mathcal{X}) \subset T_pM$. Then, the map $\deffun{f \circ \exp_p : \overline{\mathcal{X}} -> \RR;}$ is $\widehat{\alpha}_r$-weakly convex with constant
        \[
            \widehat{\alpha}_r = \alpha (C_{1,r} + C_{2,r})
        \]
        for
        \begin{gather*}
            C_{1,r} = \max\set[\Big]{1, \lfrac{\sn_\delta(r)^2}{r^2}} \\
            C_{2,r} =
                \lfrac{8}{3}\sn_\delta\pa[\big]{\lfrac{r}{2}}^2
                \pa{\Lambda\sn_\delta\pa[\big]{\lfrac{r}{2}}^2
                +2\max\set{\abs{\Delta}, \abs{\delta}}\sn_\delta\pa{r}}.
        \end{gather*}
    \end{proposition}
    \begin{proof}
        We will bound the Hessian of the map $f \circ \exp_p$ for an arbitrary $p \in M$. By the Leibnitz rule, we have that
        \[
            \conn\dif\pa{f \circ \exp_p}
            = \conn \pa{\dif f \circ \dif \exp_p}
            = \conn\dif f \circ \dif\exp_p +  \dif f \circ \conn\dif\exp_p.
        \]
        Taking norms, and since the left-hand side is a symmetric tensor, its maximum value---\ie, its norm---it's reached at a singular vector. As such, writing the bound explicitly
        \begin{align*}
            \norm{\conn\dif\pa{&f \circ \exp_p}}_{\mathcal{X}}=\\
            & \max_{\substack{v \in \exp^{-1}(\mathcal{X})\\w \in T_pM,\ \norm{w} = 1}}
            \cor[\Big]{ \pa{\conn\dif f}_{\exp_p(v)}\pa{\pa{\dif \exp_p}_v(w), \pa{\dif \exp_p}_v(w)}
            +  \pa{\dif f}_{\exp_p(v)}\pa{ \pa{\conn\dif\exp_p}_{v}(w,w)}}.
        \end{align*}
        Since there exists a critical point of $f$ in $\seg$, and by the $\alpha$-convexity, we have that $f$ is $\alpha r^2$-Lipschitz on $\mathcal{X}$. Using this, the triangle inequality and Cauchy-Schwarz, we can bound this quantity as
        \[
            \norm{\conn\dif\pa{f \circ \exp_p}}_{\mathcal{X}}=
            \alpha
            \max_{\substack{v \in \exp^{-1}(\mathcal{X})\\w \in T_pM,\ \norm{w} = 1}}
            \norm{\pa{\dif\exp_p}_v(w)}^2
            +
            \alpha r^2
            \max_{\substack{v \in \exp^{-1}(\mathcal{X})\\w \in T_pM,\ \norm{w} = 1}}
            \norm{\pa{\conn\dif\exp_p}_v(w, w)}.
        \]
        From~\Cref{thm:first_order_bounds,thm:full_second_order_bounds}, we have that $C_{1,r}, C_{2,r}$ are bounds for the square of the norm of the differential and the norm Hessian of the Hessian of the exponential respectively.
    \end{proof}

    After the heavy work of giving bounds on the Hessian of the pullback of a function along the exponential map, we are in a position to prove convergence rates for different instances of the dynamic trivialization framework in terms of $\widehat{\alpha}_r$. We start with one of the simplest ones, namely \emph{static trivializations}. This is the algorithm that comes from choosing $\code{stop} \equiv \code{False}$ in~\Cref{alg:dyn_triv}. Equivalently, this is the algorithm coming from solving
    \[
        \min_{v \in T_pM} f\pa{\exp_p(v)}
    \]
    using gradient descent on $T_pM$.

    The critical assumption in this result is that iterates remain bounded inside a compact set $\mathcal{X} \subset T_pM$. This assumption is restrictive, but it is standard in previous work~\parencite{bonnabel2013stochastic,zhang2016riemannian,sato2019riemannian,tripuraneni2018averaging,ahn2020nesterov}.

    \begin{theorem}[Convergence of static trivializations]\label{thm:static_trivializations}
        Let $(M, \gm)$ be a connected and complete Riemannian manifold of $(\delta, \Delta, \Lambda)$-bounded geometry and let $f$ be an $\alpha$-weakly convex function on it. Fix a point $p \in M$, and let $\mathcal{X} \subset \seg$ be a subset of $M$ with $\diam(\mathcal{X}) \leq r$, and at least one critical point of $f$ in it. Consider~\Cref{alg:dyn_triv} with the stopping rule $\code{stop} \equiv \code{False}$ and fixed step-size $\eta_{i,k} = \lfrac{1}{\widehat{\alpha}_r}$ where $\widehat{\alpha}_r$ is as in~\Cref{prop:weak_convexity}. If all the iterates of the method stay in $\mathcal{X}$, the method will find a point $v_{0, t} \in T_pM$ such that $\norm{\grad \pa{f \circ \exp_p}(v_{0,t})} < \epsilon$ in at most
        \[
            t \leq \ceil[\Big]{\lfrac{2\widehat{\alpha}_r}{\epsilon^2}\pa{f(\exp_p(v_{0,0})) - f^\ast}}
        \]
        steps, where $f^\ast$ is a lower-bound of $f$ on $\mathcal{X}$.
    \end{theorem}
    \begin{proof}
        We proved in~\Cref{prop:weak_convexity} that the map $f \circ \exp_p$ is $\widehat{\alpha}_r$-weakly convex on $\mathcal{X}$. This can be regarded as a function on a Euclidean space, for which the convergence rate is well known (see for example~\parencite{nesterov2004introductory}).
    \end{proof}

    As we can see, the point $p \in M$ does not play any role in the proof of~\Cref{thm:static_trivializations}, besides for the technical condition of the iterates being bounded in $\seg$. Generalizing this assumption, we can prove at once the convergence of the scheme of dynamic trivializations, for an arbitrary stopping rule.

    \begin{theorem}[Convergence of dynamic trivializations]\label{thm:dynamic_trivializations}
        Let $(M, \gm)$ be a connected and complete Riemannian manifold of $(\delta, \Delta, \Lambda)$-bounded geometry and let $f$ be an $\alpha$-weakly convex function on it. Assume that in the algorithm~\Cref{alg:dyn_triv} with an arbitrary stopping rule $\code{stop}$, all the iterates $v_{i,k}$ are contained in sets $\mathcal{X}_i \subset T_{p_i}M$ with $\diam\pa{\mathcal{X}_i} \leq r$, with at least one critical point of $f$ in each of them. Then, for the choice $\eta_{i,k} = \lfrac{1}{\widehat{\alpha}_r}$, the algorithm will find a point $v_{i, k} \in T_pM$ such that $\norm{\grad \pa{f \circ \exp_{p_i}}(v_{i,k})} < \epsilon$ in at most
        \[
            \ceil[\Big]{\lfrac{2\widehat{\alpha}_r}{\epsilon^2}\pa{f(\exp_p(v_{0,0})) - f^\ast}}
        \]
        steps, where $f^\ast$ is a lower-bound of $f$ on $\mathcal{X}$.
    \end{theorem}
    \begin{proof}
        Analogous to~\Cref{thm:static_trivializations}, as the proof does not depend on the pullback point $p$.
    \end{proof}

    This theorem should can be considered as a building block to then be used in specific examples to get actual convergence rates under weaker assumptions. In plain words, this result asserts that if the algorithm is converging to a critical point, it is fine to change the trivialization point, as the exponential map will not distort the metric too much.

    \begin{remark}[Practical considerations]
    Of course, these bounds are worst-case bounds, but one can do better in practice. As we have bounds for the distortion of the exponential map at every point, we can choose a dynamic step-length that accounts for this. For example, at a point with $\norm{v_{i,k}} = s$, we could consider
    \[
        \eta_{i,k} = \frac{1}{\widehat{\alpha}_s}.
    \]
    In practice, this is not likely to give any improvement in a real-world problem, as the weak-convexity constant of a function is itself an upper-bound on the Hessian of the function itself, and the exponential map is not a well-behaved map, as we have shown.

    This is particularly true when using a dynamic stopping rule that controls how much the norm of the gradient of the pullback deviates from the norm of the gradient of the function. In this case, one could consider having a rule similar to the one that we outlined in the introduction, but that does not only account for the case in which the algorithm is converging to the cut locus, which will mostly happen if the sectional curvature is positive, but also accounts for the case when the gradient of the pullback is much larger than that of the function, which might happen in cases of negative curvature
    \[
        \code{stop}
        \equiv
        \pa[\Big]{\frac{\norm{\grad\pa{f \circ \exp_{p_i}}\pa{v_{i, k}}}}{\norm{\grad f\pa{x_{i,k}}}} < \epsilon}
        \ 
        \code{or}
        \ 
        \pa[\Big]{\frac{\norm{\grad\pa{f \circ \exp_{p_i}}\pa{v_{i, k}}}}{\norm{\grad f\pa{x_{i,k}}}} > \frac{1}{\epsilon}}
    \]
    This rule can be read as \emph{we change the trivialization point when we detect that $\exp_p$ has deviated too much from being an isometry in the direction of the gradient}. Of course, this rule can be adapted using the information that one has a priori about the geometry of the manifold and, in particular, its cut locus, conjugate locus or injectivity radius.
    \end{remark}

\section*{Acknowledgements}

We would like to thank Jaime Mendizabal and Prof.\ Vidit Nanda for checking the proofs and giving very useful feedback on early versions of the work. We would also like to thank Prof.\ Andrew Dancer, Prof.\ Raphael Hauser and Prof.\ Coralia Cartis for useful feedback and pointers in the early stages of this project and encouragement to follow the line of work that led to this final version.

This work was supported by an Oxford-James Martin Graduate Scholarship.

\clearpage
\printbibliography
\end{document}